\documentclass[oneside, a4paper, 11pt]{amsart}
\usepackage[utf8]{inputenc}
\usepackage[T1]{fontenc}

\usepackage{amsmath,amsthm,amssymb,bbm,comment,mathrsfs}
\usepackage{mathtools}

\pdfoutput=1
\usepackage[ttscale=.85]{libertine}
\usepackage{libertinust1math}
\DeclareSymbolFont{cmarrows}{OMS}{cmsy}{m}{n}
\DeclareMathSymbol{\mapstochar}{\mathrel}{cmarrows}{"37}

\usepackage[tracking=true]{microtype}
\SetTracking% Slightly space out small caps
  {encoding=T1, shape=sc}
  {10}
\SetProtrusion% Protrude sub- and superscripts at the right margin
  {encoding=T1,size={7,8}}
  {1={ ,750},2={ ,500},3={ ,500},4={ ,500},5={ ,500},
    6={ ,500},7={ ,600},8={ ,500},9={ ,500},0={ ,500}}

\usepackage{scalerel}   % \scaleobj
\makeatletter
\let\@tl\triangleleft% Save these because we will overwrite them later.
\let\@tr\triangleright% ^
\newcommand*{\@smallTriangle}[2]{\vcenter{\hbox{\scalebox{0.75}{\ensuremath{#1#2}}}}}
\newcommand*{\@medTriangle}[2]{\vcenter{\hbox{\scalebox{0.90}{\ensuremath{#1#2}}}}}
\newcommand*{\lact}{\mathbin{\mathpalette\@smallTriangle\@tr}}
\newcommand*{\ract}{\mathbin{\mathpalette\@smallTriangle\@tl}}
\let\@btl\blacktriangleleft%
\let\@btr\blacktriangleright%
\newcommand*{\blact}{\mathbin{\mathpalette\@medTriangle\@btr}}
\newcommand*{\bract}{\mathbin{\mathpalette\@medTriangle\@btl}}
\renewcommand*{\triangleright}{\lact}
\renewcommand*{\triangleleft}{\ract}
\renewcommand*{\blacktriangleright}{\blact}
\renewcommand*{\blacktriangleleft}{\bract}

\DeclareRobustCommand{\SkipTocEntry}[5]{}

\numberwithin{equation}{section}

\numberwithin{figure}{section}

\usepackage{thmtools,mathtools}
\usepackage{graphicx,enumitem,stmaryrd}
\usepackage{multicol}
\usepackage{xparse}
\usepackage{leftidx}    % good-looking left super-/subscripts

\usepackage{anyfontsize}
\allowdisplaybreaks[4]
\usepackage[all]{xy}
\usepackage[table]{xcolor}
\usepackage[normalem]{ulem}
\usepackage{bm, tensor}
\usepackage{tikz,tikz-cd}
\usetikzlibrary{matrix}
\usetikzlibrary{fit}
\usetikzlibrary{calc}
\tikzcdset{
  arrow style=tikz,
  diagrams={>={Straight Barb}}
}

\definecolor{blue}{rgb}{0.38, 0.51, 0.71}
\definecolor{red}{RGB}{175, 49, 39}
\definecolor{green}{RGB}{146, 227, 95}
\usepackage[english]{babel}%
\usepackage[colorlinks, linktoc=page, linkcolor={red}, citecolor={green!65!black}, urlcolor={blue}]{hyperref}
\addto\extrasenglish{%
}

\makeatletter
\def\slashedarrowfill@#1#2#3#4#5{%
$\m@th\thickmuskip0mu\medmuskip\thickmuskip\thinmuskip\thickmuskip
  \relax#5#1\mkern-7mu%
  \cleaders\hbox{$#5\mkern-2mu#2\mkern-2mu$}\hfill
  \mathclap{#3}\mathclap{#2}%
  \cleaders\hbox{$#5\mkern-2mu#2\mkern-2mu$}\hfill
  \mkern-7mu#4$%
}
\def\rightslashedarrowfill@{%
\slashedarrowfill@\relbar\relbar\mapstochar\rightarrow}
\newcommand\xslashedrightarrow[2][]{%
\ext@arrow 0055{\rightslashedarrowfill@}{#1}{#2}}
\makeatother

\makeatletter
\newcommand{\ostar}{\mathbin{\mathpalette\make@circled\star}}
\newcommand{\make@circled}[2]{%
\ooalign{$\m@th#1\smallbigcirc{#1}$\cr\hidewidth$\m@th#1#2$\hidewidth\cr}%
}
\newcommand{\smallbigcirc}[1]{%
\vcenter{\hbox{\scalebox{0.77778}{$\m@th#1\bigcirc$}}}%
}
\makeatother

\usepackage{adjustbox}

\usepackage{CJKutf8}

\tikzcdset{scale cd/.style={every label/.append style={scale=#1},
  cells={nodes={scale=#1}}}}

\calclayout

\usetikzlibrary{calc}
\usetikzlibrary{decorations.pathmorphing}
\tikzset{curve/.style={settings={#1},to path={(\tikztostart)
  .. controls ($(\tikztostart)!\pv{pos}!(\tikztotarget)!\pv{height}!270:(\tikztotarget)$)
  and ($(\tikztostart)!1-\pv{pos}!(\tikztotarget)!\pv{height}!270:(\tikztotarget)$)
  .. (\tikztotarget)\tikztonodes}},
  settings/.code={\tikzset{quiver/.cd,#1}
      \def\pv##1{\pgfkeysvalueof{/tikz/quiver/##1}}},
  quiver/.cd,pos/.initial=0.35,height/.initial=0}

\tikzset{tail reversed/.code={\pgfsetarrowsstart{tikzcd to}}}
\tikzset{2tail/.code={\pgfsetarrowsstart{Implies[reversed]}}}
\tikzset{2tail reversed/.code={\pgfsetarrowsstart{Implies}}}
\tikzset{no body/.style={/tikz/dash pattern=on 0 off 1mm}}

\newtheorem{theoremm}{Theorem}[section]

\declaretheorem[style=plain,name=Theorem,numberlike=theoremm]{theorem}
\declaretheorem[style=plain,name=Theorem,numbered=no]{theorem*}

\declaretheorem[style=plain,name=Lemma,numberlike=theoremm]{lemma}
\declaretheorem[style=plain,name=Proposition,numberlike=theoremm]{proposition}
\declaretheorem[style=plain,name=Corollary,numberlike=theoremm]{corollary}

\declaretheorem[style=plain,name=Question,numberlike=theoremm]{question}
\declaretheorem[style=definition,name=Definition,numberlike=theorem]{definition}

\declaretheorem[style=remark,name=Example,numberlike=theorem]{example}
\declaretheorem[style=remark,name=Remark,numberlike=theorem]{remark}

\usepackage[cal=boondoxo, calscaled=0.96, bb=dsserif]{mathalfa}

\newcommand{\on}[1]{\operatorname{#1}}
\newcommand{\setj}[1]{\left\{ #1 \right\}}
\newcommand*{\cat}[1]{\ensuremath{\mathcal{#1}}} % A category

\DeclareFontFamily{U}{DSSerif}{\skewchar \font =45}% openface
\DeclareFontShape{U}{DSSerif}{m}{n}{<-> s*[1]  DSSerif}{}
\DeclareFontSubstitution{U}{DSSerif}{m}{n}
\DeclareMathAlphabet{\mathbbbb}{U}{DSSerif}{m}{n}

\DeclareMathAlphabet\EuRoman{U}{eur}{m}{n}
\SetMathAlphabet\EuRoman{bold}{U}{eur}{b}{n}

\newcommand*{\leftdual}[1]{\leftidx{^\vee}{\!#1}{}}     % normal left  dual
\newcommand*{\rightdual}[1]{{#1}^{\vee}}                   % normal right dual
\NewDocumentCommand{\lmod}{O{\cat{C}} O{A}}{%
  #2\raisebox{.2ex}{-}%
  \textnormal{mod}_{\raisebox{-.1ex}{\(\scriptscriptstyle#1\)}}%
}

\NewDocumentCommand{\rmod}{O{\cat{C}} O{A}}{%
  \textnormal{mod}_{\raisebox{-.1ex}{\(\scriptscriptstyle#1\)}}\hspace{-.2ex}%
  \raisebox{.2ex}{-}\hspace{-.2ex}#2%
}

\NewDocumentCommand{\modloc}{O{\cat{C}} O{A}}{%
  \textnormal{mod}^{\on{loc}}_{\raisebox{-.1ex}{\(\scriptscriptstyle#1\)}}\hspace{-.2ex}%
  \raisebox{.2ex}{-}\hspace{-.2ex}#2%
}

\NewDocumentCommand{\bimod}{O{\cat{C}} O{A}}{%
  \textnormal{bimod}_{\raisebox{-.1ex}{\(\scriptscriptstyle#1\)}}\hspace{-.2ex}%
  \raisebox{.2ex}{-}\hspace{-.2ex}#2%
}

\NewDocumentCommand{\bbimod}{O{\cat{C}} O{A} O{B}}{%
  #2\raisebox{.2ex}{-}%
  \textnormal{bimod}_{\raisebox{-.1ex}{\(\scriptscriptstyle#1\)}}\hspace{-.2ex}%
  \raisebox{.2ex}{-}\hspace{-.2ex}#3%
}

\usetikzlibrary{backgrounds}
\usetikzlibrary{arrows}
\usetikzlibrary{shapes,shapes.geometric,shapes.misc}
\tikzstyle{tikzfig}=[baseline=-0.25em,scale=0.5]
\pgfkeys{/tikz/tikzit fill/.initial=0}
\pgfkeys{/tikz/tikzit draw/.initial=0}
\pgfkeys{/tikz/tikzit shape/.initial=0}
\pgfkeys{/tikz/tikzit category/.initial=0}
\pgfdeclarelayer{edgelayer}
\pgfdeclarelayer{nodelayer}
\pgfsetlayers{background,edgelayer,nodelayer,main}
\tikzstyle{none}=[inner sep=0mm]
\tikzstyle{every loop}=[]
\tikzstyle{whitedot}=[fill=white, draw, shape=circle, scale=0.3, tikzit draw=black, tikzit shape=circle, tikzit fill=white]
\tikzstyle{blackdot}=[fill=black, draw, shape=circle, scale=0.3, tikzit draw=black, tikzit shape=circle, tikzit fill=black]
\tikzstyle{box}=[fill=white, draw=black, shape=rectangle, tikzit fill=white]
\tikzstyle{BL}=[draw=black, shape=circle, fill=black, scale=0.3]
\tikzstyle{PP}=[draw={rgb,255:red,102;green,41;blue,163}, shape=circle, fill={rgb,255:red,102;green,41;blue,163}, scale=0.3]
\tikzstyle{morphism-edge}=[-, draw=black, thick]
\tikzstyle{cotensor}=[-, draw=gray]
\tikzstyle{braid-over}=[-, draw=white, thick, double=black, double distance=0.8pt, tikzit draw={rgb,255: red,128; green,0; blue,128}]
\tikzstyle{purple-over}=[-, draw=white, thick, double={rgb,255:red,102;green,41;blue,163}, double distance=0.8pt, tikzit draw={rgb,255:red,102;green,41;blue,163}]
\tikzstyle{purple}=[-, draw={rgb,255:red,102;green,41;blue,163}, thick]
\tikzstyle{blue-under}=[-, draw={rgb,255:red,0;green,128;blue,128}, thick]
\tikzstyle{ddd}=[-, draw=black, dash dot dot]
\tikzstyle{unit}=[-, draw=black, densely dotted]
\tikzstyle{Front}=[-, draw=black, fill={rgb,255 :red,255; green,255; blue,255}, opacity=0.8]
\tikzstyle{Hidden}=[-, draw=black, fill={rgb,255 :red,255; green,255; blue,255}, opacity=0.2]
\tikzstyle{directed}=[-, thick, black, decoration={markings, mark=at position 0.5 with {\arrow{>}}}, postaction=decorate]
\newcommand{\tikzfig}[1]{{%
    \tikzstyle{every picture}=[tikzfig]
    \IfFileExists{#1.tikz}
    {\input{#1.tikz}}
    {%
      \IfFileExists{figures/#1.tikz}
      {\input{figures/#1.tikz}}
      {\tikz[baseline=-0.5em]{\node[draw=red,font=\color{red},fill=red!10!white] {\textit{#1}};}}%
    }%
}}

\usepackage[a4paper, margin=3cm]{geometry}
\begin{document}

\title{Finite Pre-Tensor Categories that are Morita Equivalent\\ to Finite Tensor Categories}

\author{Thibault Décoppet}
\address{T.D.,
Department of Mathematics, Harvard University, 1 Oxford St, Cambridge, MA 02138, USA
}
\email{decoppet@math.harvard.edu}

\author{Mateusz Stroiński}
\address{M.S.,
Fachbereich Mathematik,
Universit{\"a}t Hamburg,
%Bereich Algebra und Zahlentheorie,
Bundesstraße 55,
D-20146 Hamburg, Germany}
\email{mateusz.stroinski@uni-hamburg.de}

%\subjclass[2020]{18M05, 18M20, 16D90}

\maketitle

\begin{abstract}
A finite pre-tensor category is a finite abelian category equipped with a right exact tensor product for which every projective object has duals. Finite tensor categories, for which every object has duals, are notable examples. More generally, the category of bimodules over an algebra in a finite tensor category is a finite pre-tensor category. In particular, it is natural to extend the notion of Morita equivalence between finite tensor categories to finite pre-tensor categories.
We characterize completely those finite pre-tensor categories that are Morita equivalent to finite tensor categories. More precisely, we show that a finite pre-tensor category $\mathcal{C}$ is Morita equivalent to a finite tensor category if and only if the Drinfeld center of $\mathcal{C}$ is a finite tensor category.
We also discuss higher algebraic consequences of our characterization.
\end{abstract}

\vspace{-0.5cm}
\tableofcontents

\setlength{\parskip}{0.25em}

\addtocontents{toc}{\SkipTocEntry}
\subsection*{Acknowledgements}\label{sec:acknowledgements}
We would like to thank the organizers of the conference ``26w5586 - New Developments in Tensor Categories'' during which some of this work was carried out.
T.D.~is supported by the Simons Collaboration on Global Categorical Symmetries.
M.S.\ is supported by the Knut and Alice Wallenberg Foundation (Grant No.~2024.0339). Large parts of this work were done while M.S.\ was supported by Lundström-Åmans stipendiestiftelse.
\section{Introduction}

Categorifying the notion of the center of an algebra, the Drinfeld center is a fundamental invariant of a tensor category. As such, it is only natural to expect that the notion of Drinfeld center behaves well with that of Morita equivalence. The first result in this direction is due to Schauenburg \cite{Sch}, who proved, under faithful flatness assumptions, that the Drinfeld center of the tensor category of bimodules over an algebra coincides with the Drinfeld center of the original tensor category.
The connection between Drinfeld centers and Morita equivalence is especially strong in the theory of finite tensor categories as introduced by Etingof-Ostrik \cite{EO}.
For instance, over an algebraically closed field, two finite tensor categories are Morita equivalent if and only if their Drinfeld centers are equivalent as finite braided tensor categories \cite{OU}.

Categorifying the notion of an algebra equipped with an involution, fiat categories were introduced in \cite{MM1}.
Categories of Soergel bimodules considered in \cite[Section~7.1]{MM1} are notable examples.
A generalization, known as weakly (or quasi) fiat categories was presented in \cite{MM2}. Roughly speaking,\footnote{Technically speaking (weakly) fiat categories are strict 2-categories. We shall only be interested in the case when these 2-categories have a single object, and are therefore completely determined by the monoidal category of endomorphisms of the unique object.} a weakly fiat category is a rigid monoidal linear category with finitely many equivalence classes of indecomposable objects and finite dimensional $\on{Hom}$-spaces.
Elementary examples are given by the categories of projective
bimodules over a self-injective finite dimensional algebra.
Unlike finite tensor categories, weakly fiat categories are not required to be abelian. We can nonetheless consider their abelianizations.
While the abelianization of a weakly fiat category is always a finite monoidal category, it may fail to be a tensor category --- a priori only projective objects have duals.
The above construction thence provides examples of finite \emph{pre-tensor} categories, that is, finite monoidal categories for which every projective object has duals.

In a different direction, finite pre-tensor categories also arise naturally in the context of the cobordism hypothesis as formulated in \cite{BD,L}. More precisely, there is a putative equivalence between fully extended $n$-dimensional topological field theories valued in a target symmetric monoidal $n$-category $\mathscr{C}$ and fully dualizable objects in $\mathscr{C}$.
The cobordism hypothesis therefore produces interesting topological field theories from fully dualizable objects in symmetric monoidal higher categories.
Convenient targets are provided by higher Morita categories \cite{H, JFS}.
For instance, there is a Morita 3-category $\mathrm{Mor_1^{ten}}$ of finite tensor categories and finite bimodule categories \cite{DSPS:book}.
The fully dualizable objects therein are exactly the separable multifusion categories, and the corresponding topological field theories are expected to be closely related to those produced by the Turaev-Viro construction.
Going up in dimension, it turns out that there is simply no Morita 4-category of finite braided tensor categories \cite{BJS}. Rather, there is a Morita 4-category $\mathrm{Mor_2^{pre}}$ of finite braided pre-tensor categories.
On the one hand, the study of fully dualizable objects in $\mathrm{Mor_2^{pre}}$ has received plenty of attention \cite{BJS,BJSS,Dec:relative}.
On the other hand, the dualizability properties of the Morita 3-category $\mathrm{Mor_1^{pre}}$ of finite pre-tensor categories and finite bimodule categories have not yet been fully examined.
In particular, there is a fully faithful symmetric monoidal functor $$\mathrm{Mor_1^{ten}}\hookrightarrow \mathrm{Mor_1^{pre}}\,,$$
and it is natural to wonder whether its essential image can be characterized.
Such questions are in fact not only relevant to the abstract analysis of the dualizability of higher Morita categories, but can potentially also be used to extend the minimal non-degenerate extension theorem of \cite{JFR} beyond the semisimple case.

\subsection{Results}

Throughout, we work over a perfect field $\mathbbm{k}$.
Given a finite pre-tensor category $\mathcal{C}$, we write $\mathcal{Z}(\mathcal{C})$ for its Drinfeld center. Briefly, $\mathcal{Z}(\mathcal{C})$ is the finite braided pre-tensor category whose objects are pairs $(Z,\gamma)$ consisting of an object $Z$ of $\mathcal{C}$ equipped with a coherent half-braiding $\gamma_C:Z\otimes C\simeq Z\otimes C$ for all objects $C$ in $\mathcal{C}$. It was established in \cite{EO} that if $\mathcal{C}$ is a finite \emph{tensor} category, then $\mathcal{Z}(\mathcal{C})$ is again a finite tensor category.

A second notion that will figure prominently in our discussion is that of Morita equivalence.
Concretely, a finite $\mathcal{C}$-$\mathcal{D}$-bimodule category $\mathcal{M}$ is a Morita equivalence provided that the two canonical tensor functors
$$\mathcal{D}^{\mathrm{mop}}\rightarrow \on{End^{rex}_{\mathcal{C}}}(\mathcal{M})\quad\mathrm{and}\quad\mathcal{C}\rightarrow \on{End^{rex}_{\mathcal{D}}}(\mathcal{M})$$
are equivalences.
Succinctly, the finite $\mathcal{C}$-$\mathcal{D}$-bimodule category $\mathcal{M}$ is a Morita equivalence if the corresponding 1-morphism in $\mathrm{Mor_1^{pre}}$ is invertible. Two finite pre-tensor categories are Morita equivalent if there exists a Morita equivalence between them. This definition is a straightforward generalization of that given in the context of finite tensor categories in \cite[Section 4.1]{ENO:extension}, which is shown therein to be equivalent to the original definition introduced in \cite{EO}.
Our first contribution is to extend the theory of Morita equivalence between finite tensor categories to finite pre-tensor categories. More precisely, given a finite $\mathcal{C}$-module category $\mathcal{M}$, we completely characterize when $\mathcal{M}$ is a Morita equivalence between $\mathcal{C}$ and the monoidal opposite of $\on{End}^{\on{rex}}_{\mathcal{C}}(\mathcal{M})$.

Recall that if $\mathcal{C}$ and $\mathcal{D}$ are Morita equivalent finite tensor categories, then $\mathcal{Z}(\mathcal{C})\simeq\mathcal{Z}(\mathcal{D})$ by \cite[Corollary 3.35]{EO}.
Via an analogous argument, we find that if $\mathcal{C}$ and $\mathcal{D}$ are Morita equivalent finite pre-tensor categories, then $\mathcal{Z}(\mathcal{C})\simeq\mathcal{Z}(\mathcal{D})$.
It follows immediately from the above discussion that if $\mathcal{C}$ is a finite pre-tensor category that is Morita equivalent to a finite tensor category, then $\mathcal{Z}(\mathcal{C})$ is a finite tensor category. Our main theorem is the converse of that statement.

\begin{theorem}\label{mainthm}
Let $\mathcal{C}$ be a finite pre-tensor category. Then, $\mathcal{C}$ is Morita equivalent to a finite tensor category if and only if $\mathcal{Z}(\mathcal{C})$ is a finite tensor category.
\end{theorem}

\begin{corollary}\label{maincor1}
Let $\mathcal{C}$ be a finite pre-tensor category over an algebraically closed field. If the Drinfeld center of $\mathcal{C}$ is trivial, i.e.\ $\mathcal{Z}(\mathcal{C})\simeq\mathrm{Vec}$, then $\mathcal{C}$ is Morita equivalent to $\mathrm{Vec}$. Specifically, there is an equivalence of finite pre-tensor categories $\mathcal{C}\simeq \bimod[\mathbbm{k}][A]$ for some finite dimensional $\mathbbm{k}$-algebra $A$.
\end{corollary}

Our proof has two main steps, and makes use of the following technical property:\ We say that a finite pre-tensor category $\mathcal{C}$ is \emph{indecomposable discrete} if for every pair of projective objects $P$ and $Q$ in $\mathcal{C}$, there exists projective objects $R$ and $S$ along with an epimorphism $R\otimes P\otimes S\twoheadrightarrow Q$.
Observe that any finite tensor category is indecomposable discrete. It also follows from a straightforward argument that any finite pre-tensor category that is Morita equivalent to a finite tensor category is indecomposable discrete.
The first step of our proof consists in establishing that a finite pre-tensor category $\mathcal{C}$ is indecomposable discrete if its Drinfeld center $\mathcal{Z}(\mathcal{C})$ is rigid.
This will allow us to appeal to the machinery developed in \cite{CSZ}, or, more accurately, its generalization to indecomposable discrete finite pre-tensor categories.
Secondly, given a finite pre-tensor category $\mathcal{C}$ whose Drinfeld center is rigid, we consider the finite left $\mathcal{C}$-module category $\mathcal{C}/\mathrm{Rad}^{\mathcal{C}}(\mathcal{C})$ --- roughly speaking this is the quotient of the regular $\mathcal{C}$-module category by its radical $\mathrm{Rad}^{\mathcal{C}}(\mathcal{C})$.
By construction, the radical of $\mathcal{C}/\mathrm{Rad}^{\mathcal{C}}(\mathcal{C})$ vanishes so that $\on{End}^{\on{rex}}_{\mathcal{C}}(\mathcal{C}/\mathrm{Rad}^{\mathcal{C}}(\mathcal{C}))$ is a finite tensor category thanks to our variant of the main result of \cite{CSZ}. 
Finally, we check, using the rigidity of $\mathcal{Z}(\mathcal{C})$, that $\mathcal{C}/\mathrm{Rad}^{\mathcal{C}}(\mathcal{C})$ does exhibit a Morita equivalence between $\mathcal{C}$ and the monoidal opposite of the finite tensor category $\on{End}^{\on{rex}}_{\mathcal{C}}(\mathcal{C}/\mathrm{Rad}^{\mathcal{C}}(\mathcal{C}))$.

\subsection{Higher Categorical Applications}

As a first application of Theorem \ref{mainthm}, we can characterize completely the fully dualizable objects of the Morita 3-category $\mathrm{Mor_1^{pre}}$ of finite pre-tensor categories and finite bimodule categories. In order to express our result, recall that the Morita 3-category $\mathrm{Mor_1^{ten}}$ of finite tensor categories was studied extensively in \cite{DSPS:book}. In particular, they proved that a finite tensor category $\mathcal{C}$ is a fully dualizable object in $\mathrm{Mor_1^{ten}}$ if and only if it is \emph{separable}, in the sense that its Drinfeld center $\mathcal{Z}(\mathcal{C})$ is finite semisimple. If $\mathcal{C}$ is separable, then it is automatically finite semisimple. In fact, in characteristic zero, the converse holds \cite{DSPS:book}. However, in positive characteristic, separability is a necessary refinement of finite semisimplicity.
It follows from our results that the Morita 3-category $\mathrm{Mor_1^{pre}}$ does not contain any new fully dualizable objects.

\begin{corollary}\label{cor:fdpre}
A finite pre-tensor category is fully dualizable as an object of $\mathrm{Mor_1^{pre}}$ if and only if its Drinfeld center is a fusion category if and only if it is Morita equivalent to a separable tensor category.
\end{corollary}

\noindent Our result can be viewed as a 3-categorical version of the ``Bestiary Hypothesis'', a term coined by Scheimbauer. The 2-categorical variant, treated in \cite[Appendix A]{BDSPV}, states that the fully dualizable objects in various symmetric monoidal 2-categories of $\mathbbm{k}$-linear 2-categories all agree, and are given by finite semisimple $\mathbbm{k}$-linear categories.

Our main theorem can also be leveraged to refine our understanding of the higher Morita 4-category $\mathrm{Mor_2^{pre}}$ of finite braided pre-tensor categories introduced in \cite{BJS}. Namely, they showed that separable braided tensor categories are fully dualizable objects in $\mathrm{Mor_2^{pre}}$. In followup work \cite{BJSS}, it was proven that any finite braided tensor category that is non-degenerate, that is, has trivial symmetric center, is also fully dualizable.
Both of these results were subsequently generalized in \cite{Dec:relative}, where it was established that any finite braided tensor category whose symmetric center is separable is fully dualizable. It is natural to expect that this last characterization is sharp. With the aid of Theorem \ref{mainthm}, we come close to answering this question.

\begin{corollary}\label{thm:Mor2fullydualizbaility}
Let $\mathcal{B}$ be a finite braided pre-tensor category.
If $\mathcal{B}$ is fully dualizable as an object of $\mathrm{Mor_2^{pre}}$, then its symmetric center is separable.
\end{corollary}

As further motivation for the above abstract higher categorical statement, we note the above theorem can be employed to investigate the existence of minimal non-degenerate extensions for slightly degenerate finite braided tensor categories.
Succinctly, given a finite braided tensor category $\mathcal{B}$ whose symmetric center is the category of super vector spaces, does there always exist a non-degenerate finite braided tensor category $\mathcal{A}$ that contains $\mathcal{B}$ and such that the centralizer of $\mathcal{B}$ in $\mathcal{A}$ is the category of super vector spaces?
Over an algebraically closed field of characteristic zero and under semisimplicity assumptions, this was answered positively in \cite{JFR} using tools from the theory of braided fusion 2-categories.
Thanks to our main theorem, it is possible to extract a braided fusion 2-category from any slightly degenerate finite braided tensor category.
This is a natural first step in the study of the existence of minimal non-degenerate extensions in the non-semisimple setting, which we will tackle in future work.

\section{Preliminaries}

Throughout, we work over a perfect field $\Bbbk$. After proving a technical result about finite dimensional $\mathbbm{k}$-algebras, we review elementary properties of finite pre-tensor categories and finite module categories over them. We go on to develop the theory of Morita equivalence between finite pre-tensor categories.

\subsection{Basic Algebras}

We begin by recalling some standard facts about basic $\Bbbk$-algebras, and use them to prove a technical result that will be relevant in the sequel.

\begin{definition}
    A finite-dimensional $\mathbbm{k}$-algebra $A$ is {\it basic} if it is multiplicity-free as a left module over itself.
\end{definition}

\begin{lemma}[{\cite[Corollary~I.6.10]{ASS}}]\label{lem:Moritabasic}
  Let $A$ be a finite-dimensional $\mathbbm{k}$- algebra. There is a basic $\mathbbm{k}$-algebra that is Morita equivalent to $A$.
\end{lemma}

Given a finite dimensional $\mathbbm{k}$-algebra $A$, we use $\on{Rad}(A)$ to denote its Jacobson radical. 

\begin{proposition}[{\cite[Proposition~I.6.2]{ASS}}]\label{radicalquotient}
A finite dimensional $\mathbbm{k}$-algebra $A$ is basic if and only if the finite semisimple $\mathbbm{k}$-algebra $A/\on{Rad}(A)$ is a finite product of division $\mathbbm{k}$-algebras.
\end{proposition}

\begin{proof}
Let $P(1),\ldots,P(n)$ be an irredundant list of indecomposable projectives in $\lmod[][A]$.
    We have ${}_{A}A \cong \bigoplus_{i=1}^{n}P(i)^{\oplus m_{i}}$, with $m_i\geq 1$ for all $i$. As $P(i)\not\cong P(j)$ if $i\neq j$, it follows that
    $$A/\on{Rad}(A) \cong \prod_{i=1}^{n} \on{End}_{A-}(P(i)^{\oplus m_{i}})/\on{Rad}(\on{End}_{A-}(P(i)^{\oplus m_{i}})) \cong \prod_{i=1}^{n} \on{Mat}_{m_{i}}(\mathbb{H}_i)\,,$$
    where $\mathbb{H}_i:=\on{End}_{A-}(P(i))/\on{Rad}(\on{End}_{A-}(P(i)))$ is a division $\mathbbm{k}$-algebra. The result follows.
\end{proof}

\begin{lemma}\label{lem:technical}
    Let $A$ be a basic $\mathbbm{k}$-algebra. For any idempotent $e \in A$, we have $(1-e)Ae \subseteq \on{Rad}(A)$. 
\end{lemma}

\begin{proof}
    Consider the image of $(1-e)Ae$ under the projection $\pi: A \twoheadrightarrow A/\on{Rad}(A)$. By \autoref{radicalquotient}, we have $A/\on{Rad}(A)\cong \prod_{i=1}^{n} \mathbb{H}_i$.
    Since $\pi$ is a morphism of $\mathbbm{k}$-algebras and both $e$ and $1-e$ are idempotents, so are $\pi(e)$ and $\pi(1-e) = 1-\pi(e)$. Hence they are of the form $(e_{1},\ldots, e_{n})$ and $(1-e_{1},\ldots,1-e_{n})$, where $e_i$ is an idempotent in $\mathbb{H}_i$ for $i=1,\ldots, n$. But any division $\mathbbm{k}$-algebra has exactly two idempotents, $0$ and $1$, so that $(1-e_{i})\mathbb{H}_ie_{i} = 0$ for all $i=1,\ldots, n$. Thus, we have
    $$\pi((1-e)Ae) = \prod_{i=1}^{n} (1-e_{i})\mathbb{H}_i e_{i} = 0\,,$$ from which it follows that $(1-e)Ae \subseteq \on{Rad}(A)$.
\end{proof}

\begin{lemma}\label{radicalofcorner}
    For any idempotent $e$ in a finite dimensional $\mathbbm{k}$-algebra $A$, we have $\on{Rad}(eAe) = e\on{Rad}(A)e$.
\end{lemma}

\begin{proof}
    Since $(e\on{Rad}(A)e)^{n} \subseteq e(\on{Rad}(A)^{n})e$ for any positive integer $n$, we find that $e\on{Rad}(A)e$ is nilpotent, so that $e\on{Rad}(A)e \subseteq \on{Rad}(eAe)$. On the other hand, we have
    $$eAe/e\on{Rad}(A)e = (e+\on{Rad}(A))(A/\on{Rad}(A))(e+\on{Rad}(A))\cong \on{End}_{A/\on{Rad}(A)-}((A/\on{Rad}(A)(e+\on{Rad}(A))\,.$$
    The right hand-side is semisimple since $\lmod[][A/\on{Rad}(A)]$ is semisimple, thus $\on{Rad}(eAe) \subseteq e\on{Rad}(A)e$.
\end{proof}

\begin{proposition}\label{prop:leftexactvanishing}
    Let $A$ be basic $\mathbbm{k}$-algebra and let $e$ be an idempotent in $A$. If $(-)\otimes_{A} A/AeA:\rmod[]\rightarrow\rmod[]$ is left exact, then $eA(1-e) = 0$, so that $\on{Hom}_{-A}((1-e)A, eA)=0$. 
\end{proposition}

\begin{proof}
    Left exactness of the functor $(-)\otimes_{A} A/AeA$ is equivalent to the projectivity of the left $A$-module $A/AeA$. It follows that the following short exact sequence of left $A$-modules splits:
\[\begin{tikzcd}
	0 & AeA & A & {A/AeA} & 0
	\arrow[from=1-1, to=1-2]
	\arrow[hook, from=1-2, to=1-3]
	\arrow[two heads, from=1-3, to=1-4]
	\arrow[from=1-4, to=1-5]
\end{tikzcd}\,.\]
Choose a left $A$-module splitting $g: A \twoheadrightarrow AeA$.
For any $x \in AeA(1-e)$, we have
\begin{equation}\label{eq:x=xg1}
x = x(1-e) = g(x)(1-e) = xg(1)(1-e) = (x(1-e))(g(1)(1-e)) = x\big((1-e)g(1)(1-e)\big).
\end{equation}
The first and fourth equalities hold as $x \in AeA(1-e)$ and $AeA(1-e) \subseteq A(1-e)$, and the second as $x \in AeA(1-e)$ and $AeA(1-e) \subseteq AeA$.
The third equality follows from the fact that $g$ is a left $A$-module map, and the fifth by associativity.

Now, observe that $(1-e)g(1)(1-e) \in (1-e)AeA(1-e)$. As $(1-e)Ae \subseteq \on{Rad}(A)$ and $eA(1-e) \subseteq \on{Rad}(A)$ by \autoref{lem:technical}, it follows that $(1-e)g(1)(1-e)\in (1-e)\on{Rad}(A)(1-e)$. Then, by appealing to \autoref{radicalofcorner}, we find that $(1-e)g(1)(1-e)\in \on{Rad}((1-e)A(1-e))$.

We have seen in \autoref{eq:x=xg1} that $x\cdot \big((1-e))g(1)(1-e)\big)= x$ for any $x\in AeA(1-e)$. Given that $(1-e)g(1)(1-e) \in \on{Rad}((1-e)A(1-e))$, it follows that
\begin{equation}\label{mjnakayama}
\Big(AeA(1-e)\Big)\on{Rad}\Big((1-e)A(1-e)\Big) = AeA(1-e)\,.
\end{equation}
Writing $M$ for the right $(1-e)A(1-e)$-module $AeA(1-e)$, and $J = \on{Rad}\big((1-e)A(1-e)\big)$ for the Jacobson radical of $(1-e)A(1-e)$, \autoref{mjnakayama} asserts that $MJ=M$.
But, the ideal $J$ is nilpotent, whence $M = 0$ by Nakayama's Lemma.
Thus, we have $AeA(1-e) = 0$, and thereby also $eA(1-e) = 0$.
\end{proof}

\subsection{Finite Pre-Tensor Categories}

Fix a perfect field $\Bbbk$. We assume familiarity with the content and terminology of \cite{EGNO}. We emphasize that we follow a different convention for duals. Given an object $V$ in a monoidal category $\mathcal{C}$, a {\it right dual} to $V$ is an object $W$ of $\mathcal{C}$ together with morphisms $\eta: \mathbb{1} \rightarrow W \otimes V$ and $\varepsilon: V \otimes W \rightarrow \mathbb{1}$ satisfying the zigzag equations. If a right dual to $V$ exists, it is essentially unique. In this case, we denote the right dual to $V$ by $V^{\vee}$, and we say that $V$ is a {\it right rigid} object. We say that the monoidal category $\mathcal{C}$ is right rigid if all of its objects are right rigid. Similar considerations apply to left duals and left rigidity. Provided it exists, we denote the left dual to $V$ by ${}^{\vee}V$. We say that a monoidal category is rigid if it is both left and right rigid. Recall the following fundamental definition:

\begin{definition}
 A {\it finite multitensor category} is a finite abelian rigid monoidal $\Bbbk$-linear category. A \emph{finite tensor category} is a finite multitensor category whose monoidal unit $\mathbbm{1}$ is a simple object.
\end{definition}

We will be interested in the following more general notion which can be traced back to \cite{BJS}.

\begin{definition}
 A {\it finite pre-multitensor category} is a finite abelian monoidal $\Bbbk$-linear category whose monoidal product is right exact in both variables, and whose projective objects are rigid. A {\it finite pre-tensor category} is a finite pre-multitensor category whose monoidal unit $\mathbbm{1}$ is indecomposable.
\end{definition}

\begin{example}\label{example:bimodka}
Let $A$ be a finite-dimensional $\Bbbk$-algebra. The relative tensor product $\otimes_A$ endows the category $\bimod[\Bbbk][A]$ of $A$-$A$-bimodules with a monoidal structure. It is a folklore result that $\bimod[\Bbbk][A]$ is a finite pre-multitensor category (see e.g.\ \cite[Lemma~5.5]{FSSW} for a recent account).

In more detail, the left rigid objects of $\bimod[\Bbbk][A]$ are precisely the $A$-$A$-bimodules that are projective as left $A$-modules. Explicitly, if $M$ is such an $A$-$A$-bimodule, its left dual is given by $\on{Hom}_{A\textrm{-}}(M,A)$.
Similarly, the right rigid objects are exactly the $A$-$A$-bimodules that are projective as right $A$-modules. If $N$ is such an $A$-$A$-bimodule, its right dual is given by $\on{Hom}_{\textrm{-}A}(N,A)$.
We emphasize that these dual objects are not necessarily rigid themselves.
\end{example}

\begin{lemma}\label{lem:rigidtensorprojective}
 Let $\mathcal{C}$ be a finite pre-multitensor category. For a projective object $P$ and a right rigid object $V$ in $\mathcal{C}$, the object $V \otimes P$ is projective. Similarly, for a left rigid object $W$ in $\mathcal{C}$, the object $P \otimes W$ is projective.
\end{lemma}

\begin{proof}
 We consider the case when a right rigid object $V \in \mathcal{C}$ is given. The other is similar. We have $\on{Hom}_{\mathcal{C}}(V\otimes P,-) \cong \on{Hom}_{\mathcal{C}}(P, V^{\vee} \otimes -)$. The former functor is clearly left exact, the latter is right exact, being a composite of two right exact functors. This establishes the exactness of $\on{Hom}_{\mathcal{C}}(V\otimes P,-)$ and hence the projectivity of $V\otimes P$.
\end{proof}

\begin{corollary}\label{cor:tensorprojective}
Given any projective objects $P,Q$ in a finite pre-multitensor category $\mathcal{C}$, we have that the objects $P \otimes Q, P \otimes Q^{\vee}$ and ${}^{\vee}P \otimes Q$ are projective.
\end{corollary}

Recall from \cite[Section 1.11]{EGNO} the notion of Deligne tensor product of finite categories.
Proceeding as in \cite[Section 4.6]{EGNO}, we find that, if $\mathcal{C}$ and $\mathcal{D}$ are two finite pre-multitensor categories, then $\mathcal{C}\boxtimes\mathcal{D}$ is a finite monoidal category whose monoidal product is right exact in both variables. For completeness, we record the following observation. 

\begin{lemma}
Let $\mathcal{C}$ and $\mathcal{D}$ be two finite pre-multitensor categories. Their Deligne tensor product $\mathcal{C}\boxtimes\mathcal{D}$ is a finite pre-multitensor category.
\end{lemma}
\begin{proof}
By construction, the canonical right exact bilinear functor $F:\mathcal{C}\times\mathcal{D}\rightarrow\mathcal{C}\boxtimes\mathcal{D}$ is monoidal. If $P\in\mathcal{C}$ and $Q\in\mathcal{D}$ are projective objects, they are rigid by assumption. It follows that $(P,Q)\in\mathcal{C}\times\mathcal{D}$ is rigid, so that its image $P\boxtimes Q\in\mathcal{C}\boxtimes\mathcal{D}$ is also rigid.
But every projective object of $\mathcal{C}\boxtimes\mathcal{D}$ is a direct summand of $P\boxtimes Q$ for some projective objects $P\in\mathcal{C}$ and $Q\in\mathcal{D}$. Namely, if $A$ and $B$ are finite-dimensional $\mathbbm{k}$-algebras such that $\mathcal{C}\simeq\rmod[\mathbbm{k}][A]$ and $\mathcal{D}\simeq\rmod[\mathbbm{k}][B]$ as $\mathbbm{k}$-linear categories, then the canonical functor $F$ can be identified with the tensor product functor $\rmod[\mathbbm{k}][A]\times\rmod[\mathbbm{k}][B]\rightarrow\rmod[\mathbbm{k}][\,(A\otimes B)]$ thanks to \cite[Proposition 1.11.2]{EGNO}. This concludes the proof.
\end{proof}

\begin{remark}\label{rem:BJSrigid}
It follows from \cite[Definition-Proposition~1.3]{BJS} that a finite monoidal category $\mathcal{C}$ whose monoidal product is right exact in both variables is pre-multitensor if and only if the tensor product functor $T:\mathcal{C}\boxtimes\mathcal{C}\rightarrow\mathcal{C}$ admits a right exact right adjoint as a $\mathcal{C}$-$\mathcal{C}$-bimodule functor.
\end{remark}

\subsection{Finite Module Categories}

The notion of a left module category over an arbitrary monoidal category is unpacked in \cite[Section 7.1]{EGNO}.

\begin{definition}
Let $\mathcal{C}$ be a finite pre-tensor category.
We say that a $\mathcal{C}$-module category $\mathcal{M}$ over a $\mathcal{C}$ is {\it finite} if $\mathcal{M}$ is a finite abelian $\mathbbm{k}$-linear category and the action functor $- \lact_{\mathcal{M}} -:\mathcal{C}\times\mathcal{M}\rightarrow\mathcal{M}$ is bilinear and right exact in both variables.
\end{definition}

\begin{example}
For any algebra $A$ in a finite pre-multitensor category $\mathcal{C}$, the category $\rmod$ of right $A$-module in $\mathcal{C}$ is a finite left $\mathcal{C}$-module category.
\end{example}

Module functors between finite module categories will feature prominently in our subsequent discussion. In particular, it will be necessary to understand when the adjoint of a module functor is again a module functor. We begin by recalling some standard observations in the rigid setting.

\begin{proposition}[{Doctrinal adjunction, \cite[Theorem~1.2]{Ke}, \cite[Theorem~4.13]{HZ}}]\label{doctrinal}
 Let $\mathcal{C}$ be a monoidal category. For any two $\mathcal{C}$-module categories $\mathcal{M},\mathcal{N}$ and any pair of adjoint functors
\[\begin{tikzcd}
	{F: \mathcal{M}} & {\mathcal{N}:G\,,}
	\arrow[""{name=0, anchor=center, inner sep=0}, shift left=2, from=1-1, to=1-2]
	\arrow[""{name=1, anchor=center, inner sep=0}, shift left=2, from=1-2, to=1-1]
	\arrow["\dashv"{anchor=center, rotate=-90}, draw=none, from=0, to=1]
\end{tikzcd}\]
 there is a bijection between oplax $\mathcal{C}$-module structures on $F$ and lax $\mathcal{C}$-module structures on $G$.
\end{proposition}

\begin{proposition}[{\cite[Remark~4]{Os}, \cite[Lemma~2.10]{DSPS}}]\label{strengthdsps}
 Let $\mathcal{C}$ be a monoidal category and let $F: \mathcal{M} \rightarrow \mathcal{N}$ be a lax $\mathcal{C}$-module functor, whose lax $\mathcal{C}$-module structure is denoted by $$s_{V,X}: V \lact F(X) \rightarrow F(V\lact X)\,,$$
 for any $V\in\mathcal{C}$ and $X\in\mathcal{M}$.
 For any left rigid object $V\in\mathcal{C}$ and any $X \in \mathcal{M}$, the morphism $s_{V,X}$ is invertible.
 In particular, if $\mathcal{C}$ is left rigid, then every lax $\mathcal{C}$-module functor is strong.
\end{proposition}

\begin{corollary}\label{moduleadjoints}
 Let $\mathcal{C}$ be a rigid monoidal category, then any (left or right) adjoint to a $\mathcal{C}$-module functor is a $\mathcal{C}$-module functor.
\end{corollary}

We will employ the following generalization, which is essentially \cite[Lemma 4.2]{BJS}.

\begin{proposition}\label{prop:laxrightexactstrong}
 Let $\mathcal{C}$ be a finite pre-multitensor category, let $\mathcal{M}$ and $\mathcal{N}$ be finite $\mathcal{C}$-module categories. A right exact lax (or oplax) $\mathcal{C}$-module functor $F: \mathcal{M} \rightarrow \mathcal{N}$ is automatically a strong $\mathcal{C}$-module functor.
\end{proposition}

\begin{proof}
 We prove the lax case, the oplax case is similar. By \autoref{strengthdsps}, for any projective object $P \in \mathcal{C}$, the morphism $s_{P,X}: P \lact F(X) \rightarrow F(P \lact X)$ is invertible, by rigidity of projective objects. Now take any $V \in \mathcal{C}$ and choose a projective presentation $P_{1} \rightarrow P_{0} \twoheadrightarrow V$ of $V$ in $\mathcal{C}$. By right exactness of the actions in both variables and right exactness of $F$, the following commutative diagram has exact rows:
\[\begin{tikzcd}
	{P_{1}\lact F(X)} & {P_{0}\lact F(X)} & {V\lact F(X)} & 0 \\
	{F(P_{1} \lact X)} & {F(P_{0} \lact X)} & {F(V \lact X)} & 0
	\arrow[from=1-1, to=1-2]
	\arrow["\cong"', from=1-1, to=2-1]
	\arrow[two heads, from=1-2, to=1-3]
	\arrow["\cong"', from=1-2, to=2-2]
	\arrow[from=1-3, to=1-4]
	\arrow["{s_{V,X}}", from=1-3, to=2-3]
	\arrow[from=2-1, to=2-2]
	\arrow[two heads, from=2-2, to=2-3]
	\arrow[from=2-3, to=2-4]
\end{tikzcd}\]
 Via the five lemma, the invertibility of $s_{V,X}$ follows from the invertibility of $s_{P_{0},X}$ and $s_{P_{1},X}$.
\end{proof}

Let $\mathcal{C}$ be a finite pre-multitensor category, and let $\mathcal{M}$ be a finite left $\mathcal{C}$-module category. The functor $- \lact_{\mathcal{M}} -:\mathcal{C}\times\mathcal{M}\rightarrow\mathcal{M}$ is in particular right exact in $\mathcal{C}$ by assumption. Thus, for every $M\in\mathcal{M}$, the functor $- \lact_{\mathcal{M}} M$ admits a right adjoint, which we denote by $\underline{\on{Hom}}(M,-)$.
Said differently, for any $C\in\mathcal{C}$ and $M,N\in\mathcal{M}$, there is a natural isomorphism
\begin{equation}\label{eq:enrichedhom}
\on{Hom}_{\mathcal{M}}(C\lact M, N)\cong \on{Hom}_{\mathcal{C}}(C, \underline{\on{Hom}}(M,N))\,.
\end{equation}
The functor $\underline{\on{Hom}}(-,-):\mathcal{M}^{\on{op}}\times\mathcal{M}\rightarrow\mathcal{C}$ supplies $\mathcal{M}$ with the structure of a $\mathcal{C}$-enriched category. In particular, for any $M\in\mathcal{M}$, we have that $\underline{\on{End}}(M)$ is an algebra in $\mathcal{C}$.

\begin{definition}
An object $P\in\mathcal{M}$ is {\it $\mathcal{C}$-projective} if $\underline{\on{Hom}}(P,-):\mathcal{M}\rightarrow\mathcal{C}$ is right exact.
An object $P\in\mathcal{M}$ is a {\it $\mathcal{C}$-generator} if $\underline{\on{Hom}}(P,-):\mathcal{M}\rightarrow\mathcal{C}$ is faithful.
\end{definition}

\noindent As the next results demonstrate, the existence of $\mathcal{C}$-projective generators affords us a dictionary between finite $\mathcal{C}$-module categories and algebras in $\mathcal{C}$.

\begin{theorem}[{\cite[Theorem 2.24]{DSPS}, \cite{Os}}]\label{thm:finitemodulemodule}
Let $\mathcal{C}$ be a finite pre-multitensor category, and let $\mathcal{M}$ be a finite $\mathcal{C}$-module category. If $P$ is a $\mathcal{C}$-projective generator in $\mathcal{M}$, then the functor
$$\underline{\on{Hom}}(P,-):\mathcal{M}\rightarrow \rmod[\cat{C}][\ \underline{\on{End}}(P)]$$
is an equivalence of left $\mathcal{C}$-module categories.
\end{theorem}

\begin{lemma}[{\cite[Lemma 2.22]{DSPS}}]\label{lem:finitemodulegenerator}
    Let $\mathcal{M}$ be a finite $\mathcal{C}$-module category. Any projective generator of $\mathcal{M}$ is a $\cat{C}$-projective generator of $\cat{M}$. In particular, any finite $\cat{C}$-module category admits a $\cat{C}$-projective generator.
\end{lemma}

\subsection{Dual Pre-Tensor Categories}

As a primer towards our discussion of Morita equivalence between finite pre-multitensor categories in the next section, we discuss the properties of the monoidal category $\on{End}^{\on{rex}}_{\mathcal{C}}(\mathcal{M})$ of right exact $\mathcal{C}$-module endofunctors on $\mathcal{M}$. This monoidal category is referred to as the ``dual'' to $\mathcal{C}$ with respect to $\mathcal{M}$ in \cite{EGNO}.
Thanks to \autoref{thm:finitemodulemodule} and \autoref{lem:finitemodulegenerator}, there exists an algebra $A$ in $\mathcal{C}$ such that $\mathcal{M}\simeq\rmod$ as left $\mathcal{C}$-module categories.
Thanks to a variant of the Eilenberg-Watts theorem (see, for instance, \cite[Proposition~7.11.1]{EGNO}), this can be leveraged to study $\on{End}^{\on{rex}}_{\mathcal{C}}(\mathcal{M})$.

\begin{lemma}\label{lem:EilenbergWatts}
 For a finite pre-multitensor category $\mathcal{C}$ and any algebra $A \in \mathcal{C}$, there is a monoidal equivalence
 $$(\bimod)^{\on{mop}}\simeq\on{End^{rex}_{\mathcal{C}}}(\rmod)$$
 between the monoidal opposite of the category of $A$-$A$-bimodules in $\mathcal{C}$ and the category of right exact $\mathcal{C}$-module endofunctors on $\rmod$.
\end{lemma}

Given an algebra $A$ in a finite pre-multitensor category, we wish to characterize when the monoidal category $\bimod$ is again a finite pre-multitensor category.
In order to do so, it turns out that one must impose a condition of the underlying object of $A$ in $\mathcal{C}$ viewed as a left $\mathcal{C}\boxtimes\mathcal{C}^{\on{mop}}$-module category.
This is justified because $\mathcal{C}$ is a monoidal category, so that we may view it as a $\mathcal{C}$-$\mathcal{C}$-bimodule category. Equivalently, using the Deligne tensor product, we can think of $\mathcal{C}$ as a left $\mathcal{C}\boxtimes\mathcal{C}^{\on{mop}}$-module category.
We begin establishing the following preliminary technical result.

\begin{lemma}\label{lem:CCprojectiveexact}
Let $\mathcal{C}$ be a finite pre-multitensor category.
An object $V \in \cat{C}$ is $\mathcal{C}\boxtimes\mathcal{C}^{\on{mop}}$-projective if and only if, for every projective object $P$ of $\mathcal{C}$, the functor $V \otimes P \otimes -: \cat{C} \rightarrow \cat{C}$ is exact.
\end{lemma}
\begin{proof}
Firstly, observe that an object $V$ in $\mathcal{C}$ is $\mathcal{C}\boxtimes\mathcal{C}^{\on{mop}}$-projective if and only if $Q\otimes V\otimes P\in\mathcal{C}$ is projective for every projective objects $P,Q\in\mathcal{C}$. We will use this fact repeatedly below.

Now assume that $V \otimes P \otimes -: \cat{C} \rightarrow \cat{C}$ is exact, so that it has a left adjoint. Since its left adjoint is automatically an oplax module functor and also automatically right exact, we find that it is a strong (right) module endofunctor of $\cat{C}$ thanks to \autoref{prop:laxrightexactstrong}. 
Thus, the object $V \otimes P$ is in fact left rigid, so that there are adjunctions
    \[
    \leftdual{(V\otimes P)} \otimes - \dashv V \otimes P \otimes -\,,\quad
    -\otimes V \otimes P \dashv -\otimes  \leftdual{(V \otimes P)}\,.
    \]
For any projective object $Q$ of $\mathcal{C}$, we therefore have that
    \[
    \on{Hom}_{\cat{C}}(Q \otimes V \otimes P, -) \cong \on{Hom}_{\cat{C}}(Q, - \otimes \leftdual{(V \otimes P)})\,.
    \]
As the right hand-side is a composite of right exact functors, the left hand-side is right exact, that is, the object $P \otimes V \otimes Q$ is projective.

Conversely, let $V$ be an object of $\mathcal{C}$ such that $Q \otimes V \otimes P$ is projective for all projective objects $P$ and $Q$ of $\mathcal{C}$. In particular, if we let $Q=P_{\mathbbm{1}}$ be the projective cover of the monoidal unit, we have that $P_{\mathbb{1}} \otimes V \otimes P \otimes -$ is exact. But the functor $P_{\mathbb{1}} \otimes -:\mathcal{C}\rightarrow\mathcal{C}$ is exact and faithful, so that it reflects exactness. It follows that $P_{\mathbb{1}} \otimes V \otimes P \otimes -$ is exact if and only if $V \otimes P \otimes -$ is exact, as desired.
\end{proof}

\begin{proposition}\label{prop:bimodulepretensor}
 Let $\mathcal{C}$ be a finite pre-multitensor category, and let $A$ be an algebra in $\mathcal{C}$. Then, $\bimod$ is a finite pre-multitensor category if and only if the object $A$ is $\mathcal{C}\boxtimes\mathcal{C}^{\on{mop}}$-projective.
\end{proposition}

\begin{proof}
The forgetful functor is clearly monadic, and the associated monad is right exact, which shows that $\bimod$ is finite abelian thanks to \cite[Lemma 1.6]{DSPS}. The right exactness of $-\otimes_{A}-$ in both variables follows from commutativity of colimits with colimits. The projective objects of $\bimod$ can be characterized as summands of those of the form $A \otimes P \otimes A$ for $P \in \mathcal{C}$ projective.

It therefore only remains to understand when the $A$-bimodules $A \otimes P \otimes A$ with $P \in \mathcal{C}$ projective are rigid.
Thanks to \autoref{prop:laxrightexactstrong}, it is enough to show that the functor $(A \otimes P \otimes A) \otimes_{A} -\cong A \otimes P \otimes -:
\lmod \rightarrow \lmod$
has right exact right and left adjoints if and only if $A$ is $\mathcal{C}\boxtimes\mathcal{C}^{\on{mop}}$-projective.
But the functor $A \otimes P \otimes -$ has a left adjoint if and only if $A \otimes P \otimes -$ is left exact.
Thus, we must argue that $A \otimes P \otimes -$ is left exact if and only if $A$ is $\mathcal{C}\boxtimes\mathcal{C}^{\on{mop}}$-projective.
This follows from \autoref{lem:CCprojectiveexact}.

Finally, recall that $(A \otimes P \otimes A) \otimes_{A} -:
\lmod \rightarrow \lmod$ has a right exact right adjoint if and only if it preserves projective objects.
But, for any projective object $Q\in\mathcal{C}$, we have that $(A \otimes P \otimes A) \otimes_{A} (A \otimes Q) \cong A \otimes (P\otimes A \otimes Q)$. Now, if $A$ is $\mathcal{C}\boxtimes\mathcal{C}^{\on{mop}}$-projective, then $P\otimes A \otimes Q$ is a projective object of $\mathcal{C}$, so that $A \otimes (P\otimes A \otimes Q)$ is a projective left $A$-module. This concludes the proof.
\end{proof}

If $\mathcal{C}$ is a finite multitensor category, every object in $\mathcal{C}$ is $\mathcal{C}\boxtimes\mathcal{C}^{\on{mop}}$-projective, which implies:

\begin{corollary}
Let $\mathcal{C}$ be a finite multitensor category. For any algebra $A$ in $\mathcal{C}$, we have that $\bimod$ is a finite pre-multitensor category.
\end{corollary}

\subsection{Morita Equivalence}

The notion of Morita equivalence between finite tensor categories was introduced in \cite{EO} and studied further in \cite{ENO:extension}.
The next definition generalizes this concept to finite pre-multitensor categories.
We give a complimentary higher categorical perspective on Morita equivalence as well as additional results in \autoref{sub:Mor1} below.

\begin{definition}
Let $\mathcal{C}$ and $\mathcal{D}$ be two finite pre-multitensor categories, and let $\mathcal{M}$ be a finite $\mathcal{C}$-$\mathcal{D}$-bimodule category.
We say that $\mathcal{M}$ is a \emph{Morita equivalence} if the two canonical tensor functors
$$\mathcal{D}^{\mathrm{mop}}\rightarrow \on{End^{rex}_{\mathcal{C}}}(\mathcal{M})\quad\mathrm{and}\quad\mathcal{C}\rightarrow \on{End^{rex}_{\mathcal{D}}}(\mathcal{M})$$
are equivalences. We say that $\mathcal{C}$ and $\mathcal{D}$ are Morita equivalent if there exists a Morita equivalence between them.
\end{definition}

Let $A$ be an algebra in $\mathcal{C}$. It follows from the Eilenberg-Watts theorem, see \autoref{lem:EilenbergWatts}, that the canonical tensor functor
$$(\bimod)^{\on{mop}}\rightarrow\on{End}^{\on{rex}}_{\mathcal{C}}(\rmod)$$
is an equivalence.
Thus, in order to show that $\mathcal{C}$ and $\bimod$ are Morita equivalent, it will suffice to show that the canonical tensor functor $$\on{can}:\mathcal{C}\rightarrow \on{End}^{\on{rex}}_{\bimod}(\rmod)=:\mathcal{C}^{\star\star}_A$$ is an equivalence.
Towards this aim, we introduce the following definition that ought to be compared with \cite[Definition 7.12.9]{EGNO}.

\begin{definition}\label{def:faithful}
Let $A$ be an algebra in $\mathcal{C}$. We say that $A$ is \emph{faithful} if the canonical tensor functor $\on{can}:\mathcal{C}\rightarrow\mathcal{C}_A^{\star\star}$ is faithful.
\end{definition}

\begin{example}
If $\mathcal{C}$ is rigid, then it can be decomposed into a direct sum $\mathcal{C} = \oplus_{i=1}^n\mathcal{C}_i$ of \emph{indecomposable} multitensor categories.
In this case, $A$ is faithful if and only if, for every $i$, the object $A$ has a non-zero summand in $\mathcal{C}_i$.
In particular, if $\mathcal{C}$ is both rigid and indecomposable, then any non-zero algebra is faithful.
\end{example}

\begin{lemma}\label{lem:faithfulnesstechnical}
An algebra $A$ in $\mathcal{C}$ is faithful if and only if, for any projective generator $R\in\mathcal{C}$, the functor $A\otimes R\otimes -:\mathcal{C}\rightarrow\mathcal{C}$ faithful.
\end{lemma}
\begin{proof}
Begin by observing that the forgetful functor
$$U:\on{End}^{\on{rex}}_{\bimod}(\lmod)\rightarrow\on{End}^{\on{rex}}(\lmod)$$
is faithful as being compatible with a module structure is a property of a natural transformation.
Moreover, as $A\otimes R$ is a projective generator of $\lmod$, we can identify
$$\on{End}^{\on{rex}}(\lmod)\simeq\bimod[\mathbbm{k}][\on{End}_A(A\otimes R)]$$
via the (classical!) Eilenberg-Watts theorem.
Under this last identification, the composite $U\circ \on{can}$ is given by $C\mapsto \on{Hom}_A(A\otimes R,A\otimes R\otimes C)\cong \on{Hom}_{\mathcal{C}}(R,A\otimes R\otimes C)$. As $\on{Hom}_{\mathcal{C}}(R,-)$ is exact and faithful, we find that $A$ is faithful if and only if $A\otimes R\otimes -:\mathcal{C}\rightarrow\mathcal{C}$ is faithful.
\end{proof}

The next result generalizes \cite[Theorem 7.12.11]{EGNO}.

\begin{theorem}\label{thm:generalizedMorita}
Let $\mathcal{C}$ be a finite pre-multitensor category. Let $A$ be a faithful algebra in $\mathcal{C}$ whose underlying object is $\mathcal{C}\boxtimes\mathcal{C}^{\on{mop}}$-projective. The canonical tensor functor 
$\on{can}:\mathcal{C}\rightarrow \mathcal{C}_A^{\star\star}$ is an equivalence.
In particular, the finite pre-multitensor categories $\mathcal{C}$ and $\bimod$ are Morita equivalence.
\end{theorem}
\begin{proof}
Because of handedness, it will be more convenient to prove that the canonical tensor functor
$$\on{can}:\mathcal{C}\rightarrow\mathrm{End}^{\on{rex}}_{\bimod}(\lmod)$$
is an equivalence. Up to replacing $\mathcal{C}$ by its monoidal opposite, this gives the desired result. 
Recall that for any projective objects $R$ and $R'$ in $\mathcal{C}$, we have that $R'\otimes A\otimes R$ is projective in $\mathcal{C}$ as $A$ is $\mathcal{C}\boxtimes\mathcal{C}^{\on{mop}}$-projective.
Consequently, we have that $A\otimes R$ is a $\mathcal{C}$-projective object of $\mathcal{C}$.
We claim that $A\otimes R$ is also a $\mathcal{C}$-generator.

Said differently, we must show that $\underline{\on{Hom}}(A\otimes R,-)$ is faithful. In order to do so, it will be enough to prove that for any projective generator $R'$ of $\mathcal{C}$, we have that $\on{Hom}_{\mathcal{C}}(R'\otimes A\otimes R,-)$ is faithful.
In particular, we may take $R' = R\otimes R$.
But we have already argued that $R\otimes A\otimes R$ is a projective object of $\mathcal{C}$, so that it is a fortiori rigid. It follows that $\underline{\on{Hom}}(R\otimes A\otimes R,-)\cong -\otimes{^{\vee}(R\otimes A\otimes R)}$.
In order to show that this functor is faithful, observe that $-\otimes(R\otimes A\otimes R)$ reflects faithfulness as it is not only exact but also faithful as $-\otimes A\otimes R$ is faithful by \autoref{lem:faithfulnesstechnical}.
Moreover, the evaluation map provides us with an epimorphism
$${^{\vee}(R\otimes A\otimes R)}\otimes{(R\otimes A\otimes R)}\twoheadrightarrow \mathbbm{1}\,,$$
because $(R\otimes A\otimes R)\otimes -$ is faithful, thence reflects epimorphisms.
It follows that
$$-\otimes{^{\vee}(R\otimes A\otimes R)}\otimes{(R\otimes A\otimes R)}:\mathcal{C}\rightarrow\mathcal{C}$$
is faithful, so that $-\otimes{^{\vee}(R\otimes A\otimes R)}$ is faithful as desired.

We have shown that $A\otimes R$ is a $\mathcal{C}$-projective generator of $\mathcal{C}$. As a byproduct of \autoref{thm:finitemodulemodule}, we therefore have an equivalence of right $\mathcal{C}$-module categories
\begin{equation}\label{eq:equivalencemodules}
\underline{\on{Hom}}(A\otimes R,-):\mathcal{C}\rightarrow \rmod[\mathcal{C}][\underline{\on{End}}(A\otimes R)]\,.
\end{equation}
We assert that there is also an equivalence of left $(\bimod)$-module categories
\begin{equation}\label{eq:beloved}
\lmod \simeq \rmod[\bimod][\underline{\on{End}}(A\otimes R)]\,.
\end{equation}
Namely, the canonical morphism $A\otimes A\otimes R\rightarrow A\otimes R$ corresponds under the adjunction 
$$\on{Hom}_{\mathcal{C}}(A\otimes A\otimes R, A\otimes R)\cong \on{Hom}_{\mathcal{C}}(A, \underline{\on{Hom}}(A\otimes R,A\otimes R))$$
to a homomorphism of algebras $A\rightarrow \underline{\on{End}}(A\otimes R)$.
In particular, we can view $\underline{\on{End}}(A\otimes R)$ as an algebra in $\bimod$.
Taking the associated categories of left $A$-modules on both sides of \autoref{eq:equivalencemodules}, we therefore obtain an equivalence of left $(\bimod)$-module categories
$$\lmod
\xrightarrow{\simeq}\bbimod[\mathcal{C}][A][\underline{\on{End}}(A\otimes R)]\,,$$
which can be composed with the formal equivalence
$$\rmod[\bimod][\underline{\on{End}}(A\otimes R)]\simeq\bbimod[\mathcal{C}][A][\underline{\on{End}}(A\otimes R)]$$
so as to yield the desired equivalence in \autoref{eq:beloved}.

We have now gathered all the ingredients necessary to finish the proof of the theorem.
By \autoref{eq:beloved} and the Eilenberg-Watts theorem, we find
\begin{equation}\label{eq:doubledualbimdoulesalt}
\mathrm{End}^{\on{rex}}_{\bimod}(\lmod)\simeq \big(\bimod[\bimod][\underline{\on{End}}(A\otimes R)]\big)^{\on{mop}}\simeq\big(\bimod[\mathcal{C}][\underline{\on{End}}(A\otimes R)]\big)^{\on{mop}}
\end{equation}
as finite pre-multitensor categories.
We emphasize that the second equivalence is formal:\ It arises from the algebra homomorphism $A\rightarrow\underline{\on{End}}(A\otimes R)$.
By inspection, the canonical tensor functor $\on{can}:\mathcal{C}\rightarrow\mathcal{C}^{**}_A$ is identified under the equivalence of \autoref{eq:doubledualbimdoulesalt} with the functor $$C\mapsto \underline{\on{Hom}}(A\otimes R,A\otimes R\otimes C)\,.$$
Namely, the inverse of the equivalence of \autoref{eq:beloved} is given explicitly by
$$-\otimes_{\underline{\on{End}}(A\otimes R)}(A\otimes R):\rmod[\bimod][\underline{\on{End}}(A\otimes R)]\xrightarrow{\simeq}
\lmod\,$$
Thanks to the \autoref{eq:equivalencemodules} and the Eilenberg-Watts theorem, we also have an equivalence of finite pre-multitensor categories
$$\mathcal{C}\simeq\on{End}^{\on{rex}}_{\mathcal{C}}(\rmod[\mathcal{C}][\underline{\on{End}}(A\otimes R)])\simeq\bimod[\mathcal{C}][\underline{\on{End}}(A\otimes R)]$$ implemented by $C\mapsto\underline{\on{Hom}}(A\otimes R, A\otimes R\otimes C)$. It follows that the functor $\on{can}$ is indeed an equivalence as desired.
\end{proof}

\subsection{Drinfeld Centers}

Recall, for instance from \cite[Section 7.13]{EGNO}, that the Drinfeld center $\mathcal{Z}(\mathcal{C})$ of a monoidal category $\mathcal{C}$ is the braided monoidal category whose objects are pairs $(Z,\gamma)$ consisting of an object $Z$ of $\mathcal{C}$ together with a half-braiding $\gamma$, that is, a suitably coherent natural isomorphism $\gamma_C:Z\otimes C\cong C\otimes Z$ for any $C\in\mathcal{C}$.

The following observation is standard.

\begin{lemma}\label{lem:centerasbimodules}
There is an equivalence of monoidal categories
$$\mathcal{Z}(\mathcal{C})\simeq \on{End}^{\on{rex}}_{\mathcal{C}-\mathcal{C}}(\mathcal{C})\,.$$
\end{lemma}

Next, we generalize \cite[Corollary 3.35]{EO}, see also \cite[Theorem 3.3]{Sch} for a related result under faithful flatness hypotheses.

\begin{proposition}\label{prop:centerMorita}
Let $\mathcal{C}$ and $\mathcal{D}$ be two Morita equivalent finite pre-multitensor categories. If $\mathcal{C}$ and $\mathcal{D}$ are Morita equivalent, then $\mathcal{Z}(\mathcal{C})\simeq\mathcal{Z}(\mathcal{D})$ as braided monoidal categories.
\end{proposition}

\begin{proof}
Let $\mathcal{M}$ be a finite $\mathcal{C}$-$\mathcal{D}$-bimodule category witnessing the Morita equivalence.
Observe that there are equivalences of monoidal categories $$\mathcal{Z}(\mathcal{C})\simeq\on{End}^{\on{rex}}_{\mathcal{C}-\mathcal{C}}(\mathcal{C})\simeq \on{End}^{\on{rex}}_{\mathcal{C}-\mathcal{D}}(\mathcal{M})\,,$$
given that $\mathcal{C}\simeq \on{End}^{\on{rex}}_{-\mathcal{D}}(\mathcal{M})$ as pre-multitensor categories by hypothesis.
Namely, we can identify the monoidal category $\on{End}^{\on{rex}}_{\mathcal{C}-\mathcal{D}}(\mathcal{M})$ with the centralizer of the monoidal functor $\mathcal{C}\rightarrow \on{End}^{\on{rex}}_{-\mathcal{D}}(\mathcal{M})$.
Likewise, we find that
$$\mathcal{Z}(\mathcal{D})\simeq\on{End}^{\on{rex}}_{\mathcal{C}-\mathcal{D}}(\mathcal{M}),$$
from which the equivalence of monoidal categories follows.

In order to see that this equivalence is compatible with the braidings, let us write $\mathcal{M}\simeq\rmod$ for some algebra $A$ in $\mathcal{C}$. It follows that $\mathcal{D}\simeq \bimod$. Moreover, the equivalence of tensor categories discussed in the previous paragraph is realized explicitly by
$$\begin{tabular}{c c c}
$\mathcal{Z}(\mathcal{C}) $&$\xrightarrow{\simeq}$ &$\mathcal{Z}(\bimod)\,,$\\
$Z$ &$\mapsto$ &Z$\otimes A$
\end{tabular}$$
which is manifestly compatible with the braidings. This concludes the proof.
\end{proof}

Over an algebraically closed field, any two finite tensor categories are Morita equivalent if and only if their Drinfeld centers are equivalent as braided tensor categories \cite[Corollary 5.14]{OU}.\footnote{Their argument fixes a gap in the backward direction of \cite[Theorem 8.12.3]{EGNO}. The forward direction is \cite[Corollary 3.35]{EO}.}
\begin{question}
Over an algebraically closed field, is it true that two finite pre-tensor categories are Morita equivalent if and only if their Drinfeld centers are equivalent as braided monoidal categories?
\end{question}
\noindent It follows from our main theorem that this holds if the Drinfeld centers are finite tensor categories.

\section{Discrete Finite Pre-Tensor Categories}

We now introduce the technical notion of discreteness that will play a key role in the sequel.
In more detail, after having examined the combinatorial properties of discreteness, we show that a finite pre-multitensor category is discrete if its Drinfeld center is left rigid. In order to do so, we study quotients of finite module categories.

\subsection{\texorpdfstring{$\mathbb{Z}_+$}{Z+}-Pseudorings and Discreteness}

Let $\mathbb{Z}_+$ denote the set of nonnegative integers equipped with addition and multiplication. The notion of a $\mathbb{Z}_+$-ring, as discussed e.g.\ in \cite[Chapter 3]{EGNO}, plays a key role in the theory of finite tensor categories.
This is due to the fact that the Grothendieck rings of finite multitensor categories are naturally $\mathbb{Z}_+$-rings.
For our present purposes, we will need a nonunital variant of this notion, based on pseudorings, which are rings without unit.

\begin{definition}
Let $A$ be a pseudoring that is free as a $\mathbb{Z}$-module.
A $\mathbb{Z}_+$-basis of $A$ is a basis $\{b_j\}_{j\in J}$ such that $b_ib_j = \sum_kc_{i,j}^kb_k$ for some $c_{i,j}^k\in\mathbb{Z}_+$.
A $\mathbb{Z}_+$-pseudoring is a pseudoring equipped with a $\mathbb{Z}_+$-basis. A $\mathbb{Z}_+$-pseudoring is finite if its $\mathbb{Z}_+$-basis is finite.
\end{definition}

\begin{example}
Let $\mathcal{C}$ be a finite pre-multitensor category, and let us write $\mathrm{Proj}(\mathcal{C})$ for its full subcategory on the projective objects. As a consequence of \autoref{cor:tensorprojective}, we find that $\mathrm{Proj}(\mathcal{C})$ inherits a not necessarily unital tensor product from $\mathcal{C}$.
The split Grothendieck pseudoring $\on{Gr}(\mathrm{Proj}(\mathcal{C}))$ of $\mathrm{Proj}(\mathcal{C})$ is a finite $\mathbb{Z}_+$-pseudoring with basis consisting of the indecomposable projective objects. Such finite $\mathbb{Z}_+$-pseudorings satisfy an additional technical property.
\end{example}

\begin{definition}
Let $A$ be a $\mathbb{Z}_+$-pseudoring. We say that $A$ is a $\mathbb{Z}_+$-nearring if, for every basis element $b_i$ in $A$, there exist basis elements $b_j$ and $b_k$ such that $b_i$ appears with positive coefficient in the product $b_jb_ib_k$.
\end{definition}

\begin{example}
Let $\mathcal{C}$ be a finite pre-multitensor category. The split Grothendieck pseudoring $\on{Gr}(\mathrm{Proj}(\mathcal{C}))$ is a $\mathbb{Z}_+$-nearring. Namely, every indecomposable projective object $P$ in $\mathcal{C}$ is a direct summand of $P_{\mathbbm{1}}\otimes P\otimes P_{\mathbbm{1}}$, where $P_{\mathbbm{1}}$ denotes the projective cover of the monoidal unit $\mathbbm{1}\in\mathcal{C}$.
\end{example}

In a $\mathbb{Z}_+$-pseudoring, it is natural to consider two-sided ideals that play nicely with the basis.

\begin{definition}\label{zplusideals}
Let $A$ be a $\mathbb{Z}_+$-pseudoring. A two-sided $\mathbb{Z}_+$-ideal in $A$ is a two-sided ideal $I$ of $A$ that is spanned as a $\mathbb{Z}$-module by a subset $\{b_j\}_{j\in K}$ of the basis $\{b_j\}_{j\in J}$ of $A$.
\end{definition}

Let $A$ be a $\mathbb{Z}_+$-pseudoring. Any two-sided $\mathbb{Z}_+$-ideal $I$ in $A$ is completely determined by a subset $K\subseteq J$ of the basis $\{b_j\}_{j\in J}$ of $A$.
We use $I^c$ to denote the two-sided $\mathbb{Z}_+$-ideal in $A$ generated by $\{b_j\}_{j\in J\backslash K}$.
Explicitly, a basis vector $b_l$ lies in $I^c$ if it appears with positive coefficient in the product $b_ib_jb_k$ for some $j\in J\backslash K$ and $i,k\in J$.

\begin{definition}\label{def:combinatorialregular}
Let $A$ be a $\mathbb{Z}_+$-pseudoring.
We say that $A$ is {\it discrete} if for any two-sided $\mathbb{Z}_+$-ideal $I$ in $A$, we have $I\cap I^c=0$. We say that $A$ is \emph{indecomposable discrete} if $A$ has no nontrivial two-sided $\mathbb{Z}_+$-ideal.
\end{definition}

Observe that every indecomposable discrete $\mathbb{Z}_+$-pseudoring is a fortiori discrete.

\begin{example}
If $\mathcal{C}$ is a finite tensor category, then $\on{Gr}(\mathrm{Proj}(\mathcal{C}))$ is indecomposable discrete. In fact, if $\mathcal{C}$ is an indecomposable finite multitensor category, that is, if $\mathcal{C}$ cannot be expressed as the direct sum of two nonzero finite multitensor categories, then $\on{Gr}(\mathrm{Proj}(\mathcal{C}))$ is indecomposable discrete. This follows from \cite[Section 4.3]{EGNO}.
Yet more generally, for any finite multitensor category $\mathcal{C}$, we have that $\on{Gr}(\mathrm{Proj}(\mathcal{C}))$ is discrete. Namely, every finite multitensor category splits as a finite direct sum of indecomposable ones, and a direct product of indecomposable discrete $\mathbb{Z}_+$-pseudorings is a discrete $\mathbb{Z}_+$-pseudoring. Below, we broaden this last observation.
\end{example}

\begin{proposition}\label{prop:regulardirectsumstronglyregular}
A finite $\mathbb{Z}_+$-nearring $A$ is discrete if and only if it is a finite direct product of indecomposable discrete finite $\mathbb{Z}_+$-nearrings.
\end{proposition}
\begin{proof}
Given basis elements $b_i$ and $b_j$ of $A$, we write $b_i\leq_{LR} b_j$ if there exists basis elements $b_k$ and $b_l$ such that $b_i$ appears with positive coefficient in the product $b_kb_jb_l$.
The relation $\leq_{LR}$ is manifestly transitive.
It is also reflexive because we have assumed that $A$ is a $\mathbb{Z}_+$-nearring.
Thus, provided that $\leq_{LR}$ is symmetric, it supplies an equivalence relation.
In this case, we write $J_1,\ldots,J_k$ for the subsets of basis elements in the same equivalence class.
Let $A_j$ denote the $\mathbb{Z}$-submodule of $A$ spanned by the basis vectors in $J_j$. By construction, the $A_j$ are two-sided $\mathbb{Z}_+$-ideals of $A$, and we have $A=\oplus_{j}A_j$ as finite $\mathbb{Z}_+$-nearrings.
Additionally, by definition of the equivalence relation $\leq_{LR}$, we have that the $A_j$ are indecomposable discrete.
It will therefore suffice to show that $\leq_{LR}$ is symmetric if and only if $A$ is discrete.

If $\leq_{LR}$ is not symmetric, there must exist basis elements $b_i$ and $b_j$ such that $b_i\leq_{LR} b_j$ and $b_j\not\leq_{LR} b_i$.
Let $I$ denote the two-sided $\mathbb{Z}_+$-ideal of $A$ generated by $b_i$.
Given that $b_j\not\leq_{LR} b_i$, we must have $b_j\not\in I$. But then, as $b_i\leq_{LR} b_j$, we find that $b_i\in I^c$.
It follows that $b_i\in I\cap I^c$, so that $A$ is not discrete.

Conversely, assume that $A$ is not discrete. Let $I$ be a two-sided $\mathbb{Z}_+$-ideal such that $I\cap I^c\neq 0$. In particular, there exists a basis element $b_i \in I\cap I^c$.
Write $K\subseteq J$ for the subset corresponding to the $\mathbb{Z}_+$-basis of $I$.
As $b_i\in I^c$, there exists $j,l\in J$ and $k\in J\backslash K$ such that $b_i$ appears with positive coefficient in the product $b_jb_kb_l$.
Thus, we find that $b_i\leq_{LR} b_k$, whereas $b_k\not\leq_{LR} b_i$, showing that $\leq_{LR}$ is not symmetric.
\end{proof}

\begin{remark}
    The notation $\leq_{LR}$ indicates the cell-theoretic origin of this relation --- it is a variant of the relations used in the study of semigroups (\cite{Green}) and Hecke algebras (\cite{KL}). The use of such relations in the study of additive monoidal categories in \autoref{cellsubsection} builds on that of \cite{MM1}.
    
    In fact, this is one of the origins of our use of the terminology {\it discrete}, as a $\mathbb{Z}_{+}$-nearring is discrete if its poset of two-sided cell is discrete. Another motivation for this choice of terminology is that we require the complement of an ideal to be an ideal --- in an algebro-geometric setting, where ideals would correspond to closed subsets of the spectrum of a $\mathbb{Z}_{+}$-nearring, this would entail the discreteness of said spectrum.
\end{remark}

\subsection{Discrete Finite Pre-Multitensor Categories}\label{cellsubsection}

Let $\mathcal{C}$ be a finite pre-multitensor category.

\begin{definition}\label{def:regular}
We say that $\cat{C}$ is {\it discrete}, respectively \emph{indecomposable discrete}, if the split Grothendieck $\mathbb{Z}_+$-pseudoring of $\mathrm{Proj}(\mathcal{C})$ is discrete, respectively indecomposable discrete.
\end{definition}

Explicitly, a finite pre-multitensor category $\mathcal{C}$ is indecomposable discrete if for every projective objects $P$ and $Q$ in $\mathcal{C}$, there exists projective objects $R$ and $S$ in $\mathcal{C}$ such that $P$ is a summand of $R\otimes Q\otimes S$.

\begin{example}
We have seen above that every indecomposable finite multitensor category $\mathcal{C}$ is indecomposable discrete.
Let $A$ be an algebra in $\mathcal{C}$ as above. We assert that the finite pre-multitensor category $\bimod$ is indecomposable discrete. To see this, recall that every projective $A$-$A$-bimodule in $\mathcal{C}$ is a summand of a free $A$-$A$-bimodule on a projective object of $\mathcal{C}$.
Now, let $Q$ be a projective $A$-$A$-bimodule in $\mathcal{C}$. As $\mathcal{C}$ is rigid, the underlying object of $Q$ in $\mathcal{C}$ is projective. Thus, for any projective object $P$ in $\mathcal{C}$, there exists projective objects $R$ and $S$ in $\mathcal{C}$ such that $P$ is a summand in $R\otimes Q\otimes S$.
It follows that the free $A$-$A$-bimodule $A\otimes P\otimes A$ is a summand of the $A$-$A$-bimodule $A\otimes R\otimes Q\otimes S\otimes A$. This proves the claim. More generally, it follows that the finite pre-multitensor category $\bimod$ is discrete for any algebra $A$ in an arbitrary finite multitensor category $\mathcal{C}$.
\end{example}

The last example above illustrates the strong relationship between discreteness and the property of being Morita equivalent to a finite multitensor category. We find it instructive to examine a nonexample.

\begin{example}\label{ex:nonregular}
Let $D = \Bbbk[x]/(x^{2})$ be the two-dimensional, non-semisimple algebra of so-called {\it dual numbers}. Let $\mathcal{S}$ be the full subcategory of $\bimod[\mathbbm{k}][D]$ consisting of finite direct sums of $D$-bimodules of the form $D\otimes_{\Bbbk} D$ and $D$. Since 
\begin{equation}\label{eq:quasiidempotent}
(D \otimes_{\Bbbk} D) \otimes_{D} (D \otimes_{\Bbbk} D) \cong (D\otimes_{\Bbbk} D)^{\oplus 2},
\end{equation}
the category $\mathcal{S}$ is a monoidal subcategory of $\bimod[\mathbbm{k}][D]$.
Since $D \otimes_{\Bbbk} D$ is projective, it is rigid by \autoref{example:bimodka}, and hence $\mathcal{S}$ is rigid.
As a consequence, the category $\mathrm{PSh}(\cat{S})$ of finite-dimensional $\mathbbm{k}$-linear presheaves on $\mathcal{S}$ is a finite pre-tensor category. This construction will be reviewed in more detail in the next subsection.
In any case, observe that $\mathrm{PSh}(\cat{S}) \not\simeq \bimod[\mathbbm{k}][D]$. Moreover, by \autoref{eq:quasiidempotent}, $\mathrm{PSh}(\cat{S})$ is not discrete. Namely, we have $D \leq_{LR} D \otimes_{\Bbbk} D$, but $D \not\leq_{LR} D\otimes_{\Bbbk} D$. More general examples of this flavour are obtained by considering presheaves on the fiat categories considered in \cite[Section 7.3]{MM1}.
\end{example}

We now briefly discuss an alternative characterization of discreteness based on the notion of a bimodule Serre subcategory. By \autoref{zplusideals}, a two-sided $\mathbb{Z}_+$-ideal in $\on{Gr}(\mathrm{Proj}(\cat{C}))$ is determined by a collection of indecomposable projective objects of $\cat{C}$.
Recall that there is a bijection between $\on{Irr}(\cat{C})$, the set of (isomorphism classes of) simple objects of $\mathcal{C}$, and the set of (isomorphism classes of) indecomposable projective objects of $\mathcal{C}$. This bijection is given explicitly by sending a simple object $L$ to its projective cover $P_L$, and its inverse sends an indecomposable projective object $Q$ to its simple top $\on{top}(Q)$.
Further, any subset of simple objects of $\mathcal{C}$ determines a {\it Serre subcategory}, that is, a full subcategory $\cat{A}$ such that, for every short exact sequence in $\mathcal{C}$
$$0\rightarrow X \rightarrow Y \rightarrow Z \rightarrow 0\,,$$
we have $Y \in \cat{A}$ if and only if both $X,Z \in \cat{A}$. If $\Sigma$ is a subset of $\on{Irr}(\cat{C})$, the corresponding Serre subcategory of $\mathcal{C}$ consists of those objects all of whose composition factors lie in $\Sigma$.

Now, a simple calculation similar to \cite[Lemma~9]{AM} shows that a set $\Pi$ of indecomposable projective objects of $\cat{C}$ spans a two-sided $\mathbb{Z}_+$-ideal in $\on{Gr}(\mathrm{Proj}(\cat{C}))$ if and only if the set of simple objects $\Sigma := \setj{\,\on{top}(P) \; | \; P \in \Pi\,}$ satisfies the following condition:
\begin{equation}\label{SerreCondition}
L \text{ is not a composition factor of } \rightdual{P}\otimes L' \otimes \leftdual{Q} \text{ for any } P,Q \in \mathrm{Proj}(\cat{C}),\, L \in \Sigma,\, L' \in \on{Irr}(\cat{C}) \setminus \Sigma\,.
\end{equation}
Equivalently, the Serre subcategory generated by $\on{Irr}(\cat{C}) \setminus \Sigma$ is a $\cat{C}$-$\cat{C}$-bimodule subcategory of $\cat{C}$.
Similar observations hold for general Serre (bi)module subcategories, see e.g.\ \cite[Section~4]{MM1}.
Comparing condition~\eqref{SerreCondition} with \autoref{def:combinatorialregular}, we obtain a reformulation of the latter:

\begin{proposition}
   A finite pre-multitensor category $\cat{C}$ is discrete if for any $\cat{C}$-$\cat{C}$-bimodule Serre subcategory $\cat{A}$ of $\cat{C}$, the Serre subcategory of $\mathcal{C}$ generated by the set of simple objects $\on{Irr}(\cat{C})\setminus\on{Irr}(\cat{A})$ is a $\cat{C}$-$\cat{C}$-bimodule subcategory of $\cat{C}$.
\end{proposition}

Our next result may be thought of as a categorical version of \autoref{prop:regulardirectsumstronglyregular}, and will be the technical heart of the proof of our main theorem.

\begin{theorem}\label{thm:regularity}
Let $\mathcal{C}$ be a finite pre-multitensor category. If the Drinfeld center of $\mathcal{C}$ is rigid, then $\mathcal{C}$ decomposes as a finite direct sum of indecomposable discrete finite pre-multitensor categories. In particular, if $\mathcal{Z}(\mathcal{C})$ is rigid, then $\mathcal{C}$ is discrete.
\end{theorem}

In order to prove the last result above, we will need to appeal to the following technical construction.

\subsection{Quotients of Finite Module Categories}\label{sub:quotient}

To any finite abelian category $\mathcal{A}$, we can associate its full subcategory $\on{Proj}(\mathcal{A})$ on the projective objects. 
We can recover $\mathcal{A}$ from its subcategory of projective objects by considering the corresponding category of (finite dimensional $\mathbb{k}$-linear) presheaves.
More precisely, let us write $\on{PSh}(\on{Proj}(\mathcal{A}))$ for the category of $\mathbbm{k}$-linear functors $\on{Proj}(\mathcal{A})^{\on{op}}\rightarrow\mathrm{Vec}$.
It is well-known that the restricted Yoneda embedding $\mathcal{A}\rightarrow \on{PSh}(\on{Proj}(\mathcal{A}))$ is an equivalence (see e.g.\ the discussion directly preceding \cite[Lemma~10.23]{Str2} for additional details).

Let now $\mathcal{C}$ be a finite pre-multitensor category.
The category $\on{Proj}(\mathcal{C})$ inherits a \emph{semigroup} structure, that is, a potentially nonunital monoidal structure.
This semigroup structure on $\mathrm{Proj}(\mathcal{C})$ induces via \emph{Day convolution} a semigroup structure on $\on{PSh}(\mathrm{Proj}(\mathcal{C}))$ that is right exact in both variables.
Moreover, the Yoneda embedding is compatible with these monoidal structures, that is, there is an equivalence $\mathcal{C}\simeq \on{PSh}(\on{Proj}(\mathcal{C}))$ of finite pre-multitensor categories.

Let now $\mathcal{M}$ be a finite $\mathcal{C}$-module category. The category $\mathrm{Proj}(\mathcal{M})$ inherits a $\mathrm{Proj}(\mathcal{C})$-module structure.
As $\mathrm{Proj}(\mathcal{C})$ is only a semigroup category, this does not involve any kind of unitality condition.
As above, we have that $\mathrm{PSh}(\mathrm{Proj}(\mathcal{M}))$ is naturally a $\mathrm{PSh}(\mathrm{Proj}(\mathcal{C}))$-module under Day convolution.
Under the equivalence of finite pre-multitensor categories $\mathcal{C}\simeq\mathrm{PSh}(\mathrm{Proj}(\mathcal{C}))$, we have that the Yoneda embedding $\mathcal{M}\rightarrow\mathrm{PSh}(\mathrm{Proj}(\mathcal{M}))$ is an equivalence of finite $\mathcal{C}$-module categories.

This last observation is the starting point of our construction.

\begin{definition}
A $\mathrm{Proj}(\mathcal{C})$-stable ideal $\mathcal{I}$ in $\mathrm{Proj}(\mathcal{M})$ is a collection of subspaces
$$\mathcal{I}(P,Q)\subseteq \on{Hom}_{\mathrm{Proj}(\mathcal{M})}(P,Q)$$
for every $P,Q\in\mathrm{Proj}(\mathcal{M})$ that is closed both under composition with arbitrary morphisms in $\mathrm{Proj}(\mathcal{M})$ and under the action of $\mathrm{Proj}(\mathcal{C})$.
\end{definition}

Given a $\mathrm{Proj}(\mathcal{C})$-stable ideal $\mathcal{I}$ in $\mathrm{Proj}(\mathcal{M})$, we can consider the $\mathbbm{k}$-linear category $\mathrm{Proj}(\mathcal{M})/\mathcal{I}$ whose objects are the objects of $\mathrm{Proj}(\mathcal{M})$ and whose morphisms spaces are given by
\begin{equation}\label{eq:quotientprojective}
\on{Hom}_{\mathrm{Proj}(\mathcal{M})/\mathcal{I}}(P,Q):=\on{Hom}_{\mathrm{Proj}(\mathcal{M})}(P,Q)/\mathcal{I}(P,Q)\,,
\end{equation}
for every $P,Q\in\mathrm{Proj}(\mathcal{M})$.
Given our assumption that $\mathcal{I}$ is a $\mathrm{Proj}(\mathcal{C})$-stable ideal, it follows that $\mathrm{Proj}(\mathcal{M})/\mathcal{I}$ is a $\mathrm{Proj}(\mathcal{C})$-module category.
We also note that $\mathrm{Proj}(\mathcal{M})/\mathcal{I}$ has finite dimensional $\on{Hom}$-spaces, and, upon Cauchy completion, has finitely many equivalence classes of indecomposable objects.
It therefore follows that $$\mathcal{M}/\mathcal{I}:=\mathrm{PSh}(\mathrm{Proj}(\mathcal{M})/\mathcal{I})$$ is a finite abelian category.
Moreover, $\mathcal{M}/\mathcal{I}$ carries an action of $\mathcal{C}\simeq\mathrm{PSh}(\mathrm{Proj}(\mathcal{C}))$ given explicitly on $C\in\mathcal{C}$, $\Phi\in\mathcal{M}/\mathcal{I}$, and $Q\in\mathrm{Proj}(\mathcal{M})/\mathcal{I}$ by
$$(C\triangleright \Phi)(Q) = \int^{R\in\mathrm{Proj}(\mathcal{M})/\mathcal{I}}\on{Hom}_{\mathrm{Proj}(\mathcal{M})/\mathcal{I}}(Q,C\otimes R)\otimes_{\mathbbm{k}}\Phi(R)\,.$$
We emphasize that this action is, at least apriori, not unital.

\begin{proposition}\label{prop:quotientmodulecat}
Let $\mathcal{C}$ be a finite pre-multitensor category, and let $\mathcal{M}$ be a finite $\mathcal{C}$-module category.
The action of $\mathcal{C}$ on $\mathcal{M}/\mathcal{I}$ is unital, that is to say, $\mathcal{M}/\mathcal{I}$ is a finite $\mathcal{C}$-module category.
\end{proposition}
\begin{proof}
Write $\pi:\mathrm{Proj}(\mathcal{M})\rightarrow\mathrm{Proj}(\mathcal{M})/\mathcal{I}$ for the quotient functor. This is a $\mathrm{Proj}(\mathcal{C})$-module functor.
Arising from $\pi$, there is a pair of adjoint functors: 
\begin{equation}\label{eq:leftKan}
\begin{tikzcd}
	{L_{\pi}: \mathcal{M}} & {\mathcal{M}/\mathcal{I}:\pi^*\,.}
	\arrow[""{name=0, anchor=center, inner sep=0}, shift left=2, from=1-1, to=1-2]
	\arrow[""{name=1, anchor=center, inner sep=0}, shift left=2, from=1-2, to=1-1]
	\arrow["\dashv"{anchor=center, rotate=-90}, draw=none, from=0, to=1]
\end{tikzcd}
\end{equation}
More precisely, $\pi^*$ is the pullback functor, whereas $L_{\pi}$ is the left Kan extension of $\pi^{\on{op}}$. Explicitly, for any $Q\in\mathrm{Proj}(\mathcal{M})/\mathcal{I}$ and $M\in\mathcal{M}$, we have
$$(L_{\pi}(M))(Q) := \int^{P\in\mathrm{Proj}(\mathcal{M})}\on{Hom}_{\mathrm{Proj}(\mathcal{M})/\mathrm{I}}(Q,\pi(P))\otimes_{\mathbbm{k}}\on{Hom}_{\mathcal{M}}(P,M)\,.$$
In particular, it follows that $L_{\pi}$ is a $\mathcal{C}$-module functor as coends commute.
We claim that $\pi^*$ is fully faithful.
To see this, let $Q\in\mathrm{Proj}(\mathcal{M})$ be an additive generator.
There is an equivalence of linear categories $\mathcal{M}\simeq\rmod[\mathbbm{k}][\on{End}_{\mathcal{M}}(Q)]$. Moreover, recall that $\mathcal{I}(Q,Q)\subseteq\on{End}_{\mathcal{M}}(Q)$ is a two-sided ideal, and, because $Q$ is a generator of $\mathrm{Proj}(\mathcal{M})/\mathcal{I}$, we have
$$\mathcal{M}/\mathcal{I}\simeq\rmod[\mathbbm{k}][\on{End}_{\mathcal{M}}(Q)/\mathcal{I}(Q,Q)]\,.$$
Under these identifications, the functor $\pi^*$ is identified with the functor
$$\rmod[\mathbbm{k}][\on{End}_{\mathcal{M}}(Q)/\mathcal{I}(Q,Q)]\rightarrow\rmod[\mathbbm{k}][\on{End}_{\mathcal{M}}(Q)]$$
induced by the homomorphism of $\mathbbm{k}$-algebras. Such functors are fully faithful, so the claim follows.

As the right adjoint $\pi^*$ is fully faithful, we have that $\mathcal{M}/\mathcal{I}$ is a reflective subcategory of $\mathcal{M}$. Thus, for any object $X\in\mathcal{M}/\mathcal{I}$, we have a natural isomorphism $X\cong L_{\pi}(\pi^*(X))$.
We therefore have (natural) isomorphisms:
$$\mathbbm{1}\triangleright X\cong\mathbbm{1}\triangleright L_{\pi}(\pi^*(X))\cong L_{\pi}(\mathbbm{1}\triangleright \pi^*(X))\cong L_{\pi}(\pi^*(X))\cong X\,.$$
The second isomorphism supplied by the left $\mathcal{C}$-module structure of the functor $L_{\pi}$.
The third isomorphism comes from the fact that the action of $\mathcal{C}$ on $\mathcal{M}$ is unital by hypothesis. This concludes the proof.
\end{proof}

\subsection{Synthesis}

\begin{proof}[Proof of \autoref{thm:regularity}]
Let $\mathcal{C}$ be a finite pre-multitensor category such that $\mathcal{Z}(\mathcal{C})$ is rigid.
Let $I$ be an arbitrary nontrivial two-sided $\mathbb{Z}_+$-ideal in the split Grothendieck pseudoring of $\mathrm{Proj}(\mathcal{C})$.
We show that
\begin{equation}\label{eq:KEYregularity}
\on{Hom}_{\mathcal{C}}(Q,P) = 0\,, \textrm{ for every }Q\in I^c \textrm{ and } P\in I\,.
\end{equation}
Manifestly, this implies that $I\cap I^c = 0$, so that $\mathcal{C}$ is discrete. It follows from \autoref{prop:regulardirectsumstronglyregular} that the $\mathbb{Z}_+$-nearring $\on{Gr}(\mathrm{Proj}(\mathcal{C}))$ splits as a direct sum of indecomposable discrete $\mathbb{Z}_+$-nearrings $A_1,\ldots, A_n$.
Let $\mathcal{C}_j$ be the Serre subcategory of $\mathcal{C}$ generated by the indecomposable projective objects in $A_j$.
We have that $\mathcal{C}_j$ is an indecomposable discrete finite pre-multitensor category as $A_j$ is an indecomposable discrete finite $\mathbb{Z}_+$-nearring.
Moreover, it follows from \autoref{eq:KEYregularity} that there is an equivalence of finite pre-multitensor categories
$$\mathcal{C}\simeq\bigoplus_{j=1}^n \mathcal{C}_j\,.$$
Namely, \autoref{eq:KEYregularity} holds both for $I$ and for $I^c$!
It remains to establish \autoref{eq:KEYregularity}.

Let $I$ be a nontrivial two-sided $\mathbb{Z}_+$-ideal in the split Grothendieck pseudoring of $\mathrm{Proj}(\mathcal{C})$.
We can consider the $\on{Proj}(\mathcal{C})$-$\on{Proj}(\mathcal{C})$-stable ideal $\mathcal{I}$ of $\mathrm{Proj}(\mathcal{C})$ defined by
$$\on{Hom}_{\mathcal{I}}(P,Q):=\{f:P\rightarrow Q\,|\,f\textrm{ factors though an object in }I\}\,,$$
for all $P,Q\in \on{Proj}(\mathcal{C})$.
Given that $I$ is a two-sided ideal, it follows that $\mathcal{I}\subseteq \on{Proj}(\mathcal{C})$ is a $\on{Proj}(\mathcal{C})$-$\on{Proj}(\mathcal{C})$-stable ideal.
In particular, the quotient functor
$$\pi:\on{Proj}(\mathcal{C})\rightarrow\on{Proj}(\mathcal{C})/\mathcal{I}$$
is a $\on{Proj}(\mathcal{C})$-$\on{Proj}(\mathcal{C})$-bimodule functor. By \autoref{eq:leftKan}, we then have a pair of adjoint functor $L_{\pi}:\mathcal{C}\leftrightarrows\mathcal{C}/\mathcal{I}:\pi^*$.
Besides, recall from the proof of \autoref{prop:quotientmodulecat} that $\pi^*$ is fully faithful and exact, and that $L_{\pi}$ is a $\mathcal{C}$-$\mathcal{C}$-bimodule functor. Thence, it follows from \autoref{prop:laxrightexactstrong} that $\pi^*$ is a $\mathcal{C}$-$\mathcal{C}$-bimodule functor as well.
We may therefore consider the composite right exact $\mathcal{C}$-$\mathcal{C}$-bimodule functor
$$\pi^*\circ L_{\pi}:\mathcal{C}\rightarrow\mathcal{C}\,,$$
which we may think of as an object of $\mathcal{Z}(\mathcal{C})$ by \autoref{lem:centerasbimodules}.
But $\mathcal{Z}(\mathcal{C})$ is left rigid by assumption, so that $\pi^*\circ L_{\pi}$ has a left adjoint.
By the special adjoint functor theorem, this last property is equivalent to $\pi^*\circ L_{\pi}$ being left exact. But, the functor $\pi^*$ reflects exactness as it is faithful and exact. Thus, we have that $L_{\pi}$ is left exact.

Now, as reviewed in \autoref{lem:Moritabasic} above, we can identify $$\mathcal{C}\simeq\on{PSh}(\on{Proj}(\mathcal{C}))\simeq
\rmod[\mathbbm{k}][A]\,,\quad A:=\on{End}_{\mathcal{C}}(\oplus_j P_j)\,,$$ as abelian categories where the direct sum runs over the equivalence classes of indecomposable projective objects in $\mathcal{C}$. In particular, the finite dimensional $\mathbbm{k}$-algebra $A$ is basic.
Let us write $e$ for the sum of the primitive idempotents corresponding to the indecomposable projective objects in the two-sided $\mathbb{Z}_+$-ideal $I$.
In other words, if $K\subset I$ denote the subset corresponding to the $\mathbb{Z}_+$-basis of $I$, the right $A$-module $eA$ is identified with $\oplus_{j\in K} P_j$.
Observe that an element $a\in A$, that is an endomorphism of $\oplus_j P_j$ in $\mathcal{C}$, factors through $I$ if and only if it lies in $AeA$.
It therefore follows from \autoref{eq:quotientprojective} that
$$\on{PSh}(\on{Proj}(\mathcal{C})/\mathcal{I})\simeq\rmod[\mathbbm{k}][\,(A/AeA)]\,.$$
Moreover, the functor $\pi^*$ is then identified with the functor $\rmod[\mathbbm{k}][\,(A/AeA)]\rightarrow\rmod[\mathbbm{k}][A]$ induced by the map of $\mathbbm{k}$-algebras $A\rightarrow A/AeA$. It follows that its left adjoint $L_{\pi}$ is identified with $$(-)\otimes_AA/AeA:\rmod[\mathbbm{k}][A]\rightarrow\rmod[\mathbbm{k}][\,(A/AeA)]\,.$$
But we have argued above that $L_{\pi}$ is left exact, so that $(-)\otimes_AA/AeA$ is exact.
Thus, thanks to \autoref{prop:leftexactvanishing}, we find that $$\on{Hom}_{-A}((1-e)A, eA)=0\,.$$
On the one hand, the right $A$-module $eA$ corresponds to the projective object $\oplus_{j\in K}P_j$ in $\mathcal{C}$, that is, to the direct sum of the indecomposable projective objects in $I$.
On the other hand, the right $A$-module $(1-e)A$ corresponds to the projective object $\oplus_{j\in J\backslash K}P_j$ in $\mathcal{C}$, that is, to the direct sum of the indecomposable projective objects of $\mathcal{C}$ not in $I$.
We claim that this establishes \autoref{eq:KEYregularity}.
Namely, if $Q$ is a projective object in $I^c$, there exists projective objects $R,S$ in $\mathcal{C}$ such that $Q$ is a direct summand of $R\otimes (\oplus_{j\in J\backslash K}P_j)\otimes S$.
But we have
$$\on{Hom}_{\mathcal{C}}(R\otimes (\oplus_{j\in J\backslash K}P_j)\otimes S, P)\cong \on{Hom}(\oplus_{j\in J\backslash K}P_j, {R^{\vee}}\otimes P \otimes {^{\vee}S}) = 0\,,$$
so that $\on{Hom}_{\mathcal{C}}(Q,P)=0$ as claimed.
This concludes the proof.
\end{proof}

\section{Exact Module Categories}

\subsection{Definition and Elementary Properties}

We now turn our attention to so-called exact module categories, first introduced in \cite{EO} in the setting of finite tensor categories. Throughout, we fix a finite pre-multitensor category $\mathcal{C}$.

\begin{definition}
A finite $\mathcal{C}$-module category $\mathcal{M}$ is said to be {\it exact} if, for any projective object $P \in \mathcal{C}$ and any arbitrary object $M \in \mathcal{M}$, the object $P \lact M$ is projective in $\mathcal{M}$.
\end{definition}

\begin{example}
 Let $\mathcal{C}$ be a finite multitensor category and let $A$ be an algebra in $\mathcal{C}$. The category $\lmod$ of left $A$-modules in $\mathcal{C}$ is an exact left module category over $\bimod$, with action given by ${{}_{A}M_{A}} \lact {{}_{A}N} := M \otimes_{A} N$. Indeed, any projective object in $\bimod$ is a direct summand of $A \otimes P \otimes A$ for $P$ projective in $\mathcal{C}$.
 Moreover, we have $(A \otimes P \otimes A) \otimes_{A} N \cong A \otimes P \otimes N$, which is projective since $P \otimes N$ is a projective object as $\mathcal{C}$ is rigid.
\end{example}

\begin{proposition}\label{exactetymology}
 Let $\mathcal{C}$ be a finite pre-multitensor category, let $\mathcal{M}$ be an exact $\mathcal{C}$-module category and let $\mathcal{N}$ be an abelian $\mathcal{C}$-module category. Any $\mathcal{C}$-module functor $F: \mathcal{M} \rightarrow \mathcal{N}$ is exact.
\end{proposition}

\begin{proof}
 Let $P$ be a projective object in $\mathcal{C}$ such that there exists an epimorphism $\pi: P \twoheadrightarrow \mathbb{1}$. Given that there is an epimorphism $\pi \lact_{\mathcal{N}}-: P \lact_{\mathcal{N}} - \twoheadrightarrow \on{Id}_{\mathcal{N}}$, the functor $P \lact_{\mathcal{N}}-$ is faithful.
 Moreover, the functor $P \lact_{\mathcal{N}} -$ admits both a left and a right adjoint, so that it is exact.
 Since $P \lact_{\mathcal{M}}-$ is exact and maps any object to a projective object by assumption, it sends any exact sequence in $\mathcal{M}$ to a split exact sequence therein. Thus, the functor $F\circ (P \lact_{\mathcal{M}} -)$ is exact as $F$ preserves split exact sequences. Since $F$ is a $\mathcal{C}$-module functor, we have $F \circ (P \lact_{\mathcal{M}} -) \cong (P \lact_{\mathcal{N}} -)\circ F$, from which we find that $F$ is exact. Namely, $P \lact_{\mathcal{N}}-$ reflects exactness as it is exact and faithful.
\end{proof}

\begin{proposition}\label{prop:dualtensorcat}
Let $\mathcal{C}$ be a finite pre-multitensor category and let $\mathcal{M}$ be an exact $\mathcal{C}$-module category. Then $\on{End}^{\on{rex}}_{\mathcal{C}}(\mathcal{M})$ is a finite multitensor category.
\end{proposition}

\begin{proof}
That $\on{End}^{\on{rex}}_{\mathcal{C}}(\mathcal{M})$ is a finite monoidal category with right exact tensor product follows from the first paragraph of the proof of \autoref{prop:bimodulepretensor} thanks to \autoref{lem:finitemodulegenerator}.
Let $F$ be a right exact $\mathcal{C}$-module endofunctor on $\mathcal{M}$. Then, by \autoref{exactetymology}, we find that $F$ is exact. Hence it admits a left adjoint $G$, which is right exact. Thus, by \autoref{doctrinal}, there is a canonical oplax $\mathcal{C}$-module structure on $G$, which is automatically strong thanks to \autoref{prop:laxrightexactstrong}. This shows that $F$ is left rigid.

We now want to show that the right adjoint of $F$ is right exact. In order to do so, it will suffice to show that $F$ sends projective objects to projective objects. Clearly, it suffices to show that there is a projective generator $Q$ in $\mathcal{M}$ such that $F(Q)$ is projective. 
Let $P \in \mathcal{C}$ be a projective object such that there exists an epimorphism $P \twoheadrightarrow \mathbb{1}$, and let $Q'$ be a projective generator for $\mathcal{M}$. Then $P \lact Q'$ is a projective generator for $\mathcal{M}$, and $F(P \lact Q') \cong P \lact F(Q')$, where the latter is a projective object by exactness of $\mathcal{M}$. Thus, the right adjoint of $F$ is right exact and hence its induced lax $\mathcal{C}$-module structure arising from \autoref{doctrinal} is actually strong thanks to \autoref{prop:laxrightexactstrong}. This establishes the right rigidity of $F$.
\end{proof}

\subsection{Radicals and Exactness}

We succinctly discuss a generalization of the main results of \cite{CSZ} to which we refer the reader for additional details. Firstly, given an additive $\mathbbm{k}$-linear category $\mathcal{A}$. Recall that its \emph{radical}, denoted by $\on{Rad}(\mathcal{A})$, is the unique ideal in $\mathcal{A}$ such that, for every object $A$ of $\mathcal{A}$ we have that $\on{End}_{\on{Rad}(\mathcal{A})}(A)\subseteq \on{End}_{\mathcal{A}}(A)$ is the Jacobson radical.

\begin{definition}
Let $\mathcal{C}$ be a finite pre-multitensor category, and let $\mathcal{M}$ be a finite $\mathcal{C}$-module category. The $\mathcal{C}$-module radical of $\mathcal{M}$ is the ideal $\on{Rad}^{\mathcal{C}}(\mathcal{M})$ in $\on{Proj}(\mathcal{M})$ defined by
$${\on{Rad}^{\mathcal{C}}(\mathcal{M})}(P,Q):=\{f\in\on{Hom}_{\mathcal{M}}(P,Q)\,|\,R\otimes f\in \on{Rad}(\mathcal{M})\mathrm{\ for\ all\ }R\in\on{Proj}(\mathcal{C})\}\,.$$
\end{definition}

\noindent Note that, by construction, the ideal $\on{Rad}^{\mathcal{C}}(\mathcal{M})$ is stable under the left action of $\on{Proj}(\mathcal{C})$ on $\on{Proj}(\mathcal{M})$.

The above definition of $\mathcal{C}$-module radical leads to a useful notion of radical for algebras within a finite pre-multitensor category. The connection is achieved through the next result.

\begin{proposition}[{\cite[Theorem 5.9]{CSZ}}]
There is an isomorphism between the lattice of $\on{Proj}(\mathcal{C})$-stable ideals in $\on{Proj}(\rmod)$ and that of two-sided ideals in $A$, which respects multiplication of ideals.
\end{proposition}

\begin{proof}
 All the results of \cite[Section~4, Section~5]{CSZ} carry over to the pre-multitensor setting. The only proofs in \cite{CSZ} that require modification are those of \cite[Lemma~5.3]{CSZ} and \cite[Proposition~5.7]{CSZ}. In both of these proofs, one uses the fact that for any non-zero projective object $P$ of a finite tensor category $\cat{D}$, the morphism $\eta: \leftdual{P} \otimes P \rightarrow \mathbb{1}$ is an epimorphism. This fails already for finite multitensor categories.
 The following modification can be used to correct both proofs. Replace the zigzag associated to the chosen projective object $P$ with the morphism obtained by first including $P$ as a direct summand of a projective generator $Q$ of the finite pre-multitensor category $\cat{C}$, continuing with the analogous zigzag for $Q$, and finally projecting back to $P$ using a section of the aforementioned inclusion. While this composite evaluates, much like the zigzag itself, to the identity, it allows us to conclude that the composite morphisms used in the proofs of \cite[Lemma~5.3]{CSZ} and \cite[Proposition~5.7]{CSZ} have the desired form.

Additionally, it is important to verify that all the objects in the domains and codomains of the $\on{Hom}$-spaces appearing in the proofs of \cite{CSZ} remain projective in the pre-multitensor case --- for example, the proof of \cite[Lemma~5.3]{CSZ} contains an unfortunate typo:\ the relation
$\xi \in \mathfrak{J}(P' \otimes Q, P' \otimes \leftdual{P'})$ should be corrected to $\xi \in \mathfrak{J}(P' \otimes Q, P' \otimes \rightdual{P'})$, as the morphism $\xi$ defined in \cite[Figure 1]{CSZ} is obtained via precomposition with the coevaluation morphism $ \epsilon: \mathbb{1} \rightarrow \rightdual{(P')}\otimes P'$. Fortunately, the objects appearing in these domains and codomains are always of the form $\leftdual{P} \otimes Q$ or $P \otimes \rightdual{Q}$ with $P,Q$ projective objects of $\mathcal{C}$, and such tensor products are projective by \autoref{cor:tensorprojective}.
\end{proof}

\begin{definition}\label{def:radicalalgebra}
Let $A$ be an algebra in a finite pre-multitensor category $\mathcal{C}$. The radical of $A$ is the two-sided ideal $\on{Rad}(A)\subseteq A$ corresponding to the $\on{Proj}(\mathcal{C})$-stable ideal $\on{Rad}^{\mathcal{C}}(\rmod)$.
\end{definition}

\noindent As in \cite[Proposition 6.12]{CSZ}, one shows that $\on{Rad}^{\mathcal{C}}(\mathcal{M})$ is the largest nilpotent $\on{Proj}(\mathcal{C})$-stable ideal in $\on{Proj}(\mathcal{M})$. It follows that the radical of $A$ is also the largest nilpotent two-sided ideal of $A$.

The next result explains the relationship between exactness and the vanishing of the module radical. We wish to draw the attention of the reader to the discreteness hypothesis featured in the statement of the next result.

\begin{theorem}[{\cite[Theorem 6.16]{CSZ}}]\label{thm:CSZpretensor}
Let $A$ be an algebra in a discrete finite pre-multitensor category $\mathcal{C}$. Then, $A$ is exact if and only if its radical is trivial.
\end{theorem}
\begin{proof}
Thanks to \autoref{thm:regularity}, it is enough to consider the case when $\mathcal{C}$ is indecomposable discrete.
Under this hypothesis, all the arguments in \cite[Section 6]{CSZ} go through seamlessly.
The only part that is not formal is the generalization of \cite[Proposition 6.13]{CSZ} because it relies on \cite[Lemma 4.16]{Str}.
This is however not a cause for concern because \cite[Lemma 4.16]{Str} generalizes without difficulty to $\mathrm{Proj}(\mathcal{C})$ provided that $\mathcal{C}$ is indecomposable discrete, as in that case, the category $\mathrm{Proj}(\cat{C})$ is {\it $\cat{J}$-cell trivial} in the terminology of \cite[Definition~4.15]{Str}, which is the only assumption of \cite[Lemma~4.16]{Str} crucially used in the result it builds on, namely \cite[Proposition~18(i)]{KM}.
\end{proof}

Given that $\on{Rad}^{\mathcal{C}}(\mathcal{M})$ is a $\mathrm{Proj}(\mathcal{C})$-stable ideal in $\mathrm{Proj}(\mathcal{M})$, we can consider the corresponding quotient $\mathcal{M}/\on{Rad}^{\mathcal{C}}(\mathcal{M})$, as in \autoref{sub:quotient}, which is a finite $\mathcal{C}$-module category thanks to \autoref{prop:quotientmodulecat}.

\begin{corollary}\label{cor:exactC/Jl}
Let $\mathcal{C}$ be a discrete finite pre-multitensor category, and let $\mathcal{M}$ be a finite $\mathcal{C}$-module category. The $\mathcal{C}$-module category $\mathcal{M}/\on{Rad}^{\mathcal{C}}(\mathcal{M})$ is exact.
\end{corollary}

\begin{example}
We warn the reader that \autoref{cor:exactC/Jl} fails if $\mathcal{C}$ is not discrete. Recall the non-discrete finite pre-tensor category $\mathrm{PSh}(\mathcal{S})$ from \autoref{ex:nonregular}. By construction, its full subcategory on the projective objects is given by $\mathcal{S}$. The left $\mathcal{S}$-module radical of $\mathcal{S}$ consists of those morphisms which remain radical upon tensoring with $D \otimes D$ on the left. In particular, not all of the morphisms in $\on{Hom}_{\mathcal{S}}(D,D\otimes D)$ lie in $\on{Rad}^{\mathcal{S}}(\mathcal{S})$. The $\mathcal{S}$-module subcategory $\on{add}\setj{D\otimes D}$ of $\mathcal{S}/\on{Rad}^{\mathcal{S}}(\mathcal{S})$ thus gives rise to a module Serre subcategory of $\mathrm{Psh}(\mathcal{S}/\on{Rad}^{\mathcal{S}}(\mathcal{S}))$, for which the left adjoint of the inclusion is a (strong) module functor that is not left exact. 
Here, the non-vanishing of $\on{Hom}_{\mathcal{S}/\on{Rad}^{\mathcal{S}}(\mathcal{S})}(D,D\otimes D)$ entails the non-exactness of the inclusion functor by an argument similar to that appearing in the proof of \autoref{thm:regularity}, using \autoref{prop:leftexactvanishing}. This shows that the $\mathrm{PSh}(\mathcal{S})$-module category $\mathrm{Psh}(\mathcal{S}/\on{Rad}^{\mathcal{S}}(\mathcal{S}))$ is not exact.
\end{example}

\begin{remark}
The example above illustrates that \autoref{thm:CSZpretensor} fails in the absence of discreteness.
Even in the presence of discreteness, we want to draw the reader's attention to the fact that the full extent of \cite[Theorem 7.1]{CSZ} does not, at least a priori, hold.
More precisely, \cite[Theorem 7.1]{CSZ} establishes that an algebra in a finite (multi)tensor category is exact if and only if it is finite semisimple.
It is not clear that this holds for algebras in discrete finite pre-multitensor categories. Specifically, the proof that exactness implies finite semisimplicity given in \cite[Proposition~7.4]{CSZ} relies on the fact that, in an exact module category $\cat{M}$ over a multitensor category $\cat{D}$, for any morphism $f \in \cat{M}$ and $Q \in \mathrm{Proj}(\cat{D})$, the morphism $Q \lact f$ is split.
This last property holds for discrete pre-multitensor categories in which the duals of projective objects are again projective, but it is not clear to us whether it holds in an arbitrary discrete pre-multitensor category.
That being said, it does follow from \autoref{thm:CSZpretensor} that any finite semisimple algebra in a discrete finite pre-multitensor category is exact.
\end{remark}

Finally, we record the following technical result, which will be used subsequently.

\begin{lemma}\label{lem:bimoduletensorradical}
Let $\mathcal{C}$ be a finite pre-multitensor category, and let $A$ be an algebra in $\mathcal{C}$. If $\on{bimod}_{\mathcal{C}}(A)$ is a finite multitensor category, then the radical of $A$ is trivial.
\end{lemma}
\begin{proof}
Write $J:=\on{Rad}(A)$ for the radical of $A$. The canonical $\mathcal{C}$-module functor 
$\rmod[\cat{C}][A/J]\rightarrow\rmod$ has a left adjoint given by $-\otimes_A A/J$. 
Observe that the $\mathcal{C}$-module functor $-\otimes_A A/J:\rmod\rightarrow\rmod$ is exact as $\bimod$ is a finite multitensor category by hypothesis.
In particular, the image of the short exact sequence $$0\rightarrow J\rightarrow A\rightarrow A/J\rightarrow 0$$ under the functor $-\otimes_A A/J$ is the short exact sequence $$0\rightarrow J/J^2\rightarrow A/J\rightarrow A/J\rightarrow 0\,.$$ The second map is an isomorphism, so that this latter sequence is exact if and only if $J/J^2 = 0$. As $J$ is nilpotent by a variant of \cite[Proposition 6.12]{CSZ}, this holds only if $J=0$.
\end{proof}

\begin{corollary}\label{cor:characterizationexactness}
Let $\mathcal{C}$ be a discrete finite pre-multitensor category. A finite $\mathcal{C}$-module category $\mathcal{M}$ is exact if and only if $\on{End}^{\on{rex}}_{\mathcal{C}}(\mathcal{M})$ is a finite multitensor category.
\end{corollary}
\begin{proof}
The forward direction is \autoref{prop:dualtensorcat}. The backward direction is \autoref{lem:bimoduletensorradical} thanks to \autoref{thm:CSZpretensor}.
\end{proof}

\subsection{Proof of \autoref{mainthm}}

We are now ready to prove the backward direction of \autoref{mainthm}. Towards this end, fix $\mathcal{C}$ a discrete finite pre-multitensor category.
We begin by observing that the rigidity of the Drinfeld center of $\mathcal{C}$ is equivalent to the vanishing of the $\mathcal{C}\boxtimes\mathcal{C}^{\on{mop}}$-module radical of $\mathcal{C}$.
For the sake of brevity, we use $\on{Rad}^{lr}(\mathcal{C})$ to denote the radical of $\mathcal{C}$ as a $\mathcal{C}\boxtimes\mathcal{C}^{\on{mop}}$-module category.

\begin{proposition}\label{prop:twosidedradicalcenter}
Let $\mathcal{C}$ be a discrete finite pre-multitensor category. We have that $\on{Rad}^{lr}(\mathcal{C})=0$ if and only if $\mathcal{Z}(\mathcal{C})$ is rigid.
\end{proposition}
\begin{proof}
By \autoref{lem:centerasbimodules}, we have that $\mathcal{Z}(\mathcal{C})\simeq \on{End}^{\on{rex}}_{\mathcal{C}-\mathcal{C}}(\mathcal{C})$.
But the $\mathcal{C}$-$\mathcal{C}$-stable ideal $\on{Rad}^{lr}(\mathcal{C})$ is by definition the radical of $\mathcal{C}$ as a left $\mathcal{C}\boxtimes\mathcal{C}^{\on{mop}}$-module category.
In particular, if $\on{Rad}^{lr}(\mathcal{C})=0$, then $\mathcal{C}$ is exact as a left $\mathcal{C}\boxtimes\mathcal{C}^{\on{mop}}$-module category thanks to \autoref{thm:CSZpretensor}. It then follows from \autoref{prop:dualtensorcat} that $\mathcal{Z}(\mathcal{C})\simeq \on{End}^{\on{rex}}_{\mathcal{C}-\mathcal{C}}(\mathcal{C})$ is a finite multitensor category.

Conversely, it follows from \autoref{lem:finitemodulegenerator} that there exists an algebra $A$ in $\mathcal{C}\boxtimes\mathcal{C}^{\on{mop}}$ such that $\mathcal{C}\simeq\rmod[\mathcal{C}\boxtimes\mathcal{C}^{\on{mop}}][A]$. By the Eilenberg-Watts theorem, we therefore have 
$$\mathcal{Z}(\mathcal{C})\simeq\on{End}^{\on{rex}}_{\mathcal{C}-\mathcal{C}}(\mathcal{C})\simeq(\bimod[\mathcal{C}\boxtimes\mathcal{C}^{\on{mop}}][A])^{\on{mop}}\,,$$
as finite pre-multitensor categories.
Given that we have assumed that $\mathcal{Z}(\mathcal{C})$ is rigid, \autoref{lem:bimoduletensorradical} establishes that the radical of $A$ is trivial. By definition (see \autoref{def:radicalalgebra}), this is equivalent to the vanishing of $\on{Rad}^{lr}(\mathcal{C})$, thereby concluding the proof.
\end{proof}

From now on, we will assume that $\mathcal{C}$ is a discrete finite pre-multitensor category such that $\on{Rad}^{lr}(\mathcal{C})=0$ or, equivalently thanks to \autoref{prop:twosidedradicalcenter}, whose Drinfeld center is rigid.
In order to establish the backward direction of \autoref{mainthm}, we will show that the exact left $\mathcal{C}$-module $\mathcal{C}/\on{Rad}^{\mathcal{C}}(\mathcal{C})$ witnesses a Morita equivalence between $\mathcal{C}$ and the monoidal opposite of $\on{End}_{\mathcal{C}}^{\on{rex}}(\mathcal{C}/\on{Rad}^{\mathcal{C}}(\mathcal{C}))$. 
We begin by establishing the following technical result.

\begin{proposition}\label{prop:Matmagic}
For any projective objects $P,Q \in \mathcal{C}$ and any $f \in \on{Hom}_{\mathcal{C}}(P,Q)$, the morphism $f$ is sent to zero by the canonical tensor functor
$$\begin{tabular}{rccc}
  $\on{can}:$ & $\mathcal{C}$ & $\rightarrow$ & $\on{End}^{\on{rex}}_{\on{End}_{\mathcal{C}}^{\on{rex}}(\mathcal{C}/\on{Rad}^{\mathcal{C}}(\mathcal{C}))}(\mathcal{C}/\on{Rad}^{\mathcal{C}}(\mathcal{C}))\,.$\\
  & $C$ & $\mapsto$ & $C \triangleright -$
\end{tabular}$$
if and only if $f \in \on{Rad}^{lr}(\mathcal{C})$.
\end{proposition}

\begin{proof}
 Unpacking the definition, for any projective objects $Q,Q'$ in $\mathcal{C}$, we have:
 \[
      \on{Hom}_{\on{Rad}^{\mathcal{C}}(\mathcal{C})}(Q,Q') = \setj{f \in \on{Hom}_{\mathcal{C}}(Q,Q') \; | \; P \otimes f \text{ is radical, for all }P \in \on{Proj}(\mathcal{C})}\,,
 \]
 \[
      \on{Hom}_{\on{Rad}^{lr}(\mathcal{C})}(Q,Q') = \setj{f \in \on{Hom}_{\mathcal{C}}(Q,Q') \; | \; P \otimes f \otimes P' \text{ is radical, for all }P, P' \in \on{Proj}(\mathcal{C})}\,.
 \]

Let $f:Q\rightarrow Q'$ be a morphism in $\mathcal{C}$.
Since $f \otimes - = 0$ implies $f = 0$, for $f$ to be in the kernel of the canonical tensor functor $\on{can}$, it must satisfy $f \otimes g \in \on{Rad}^{\mathcal{C}}(\mathcal{C})$ for all morphisms $g$. Equivalently, we have $f \otimes P' \in \on{Rad}^{\mathcal{C}}(\mathcal{C})$ for all $P'$. Spelling out this condition, we find that, for all $P'\in \on{Proj}(\mathcal{C})$, we have $P \otimes (f \otimes P')$ is radical, for all $P \in \on{Proj}(\mathcal{C}).$
  This is exactly saying that $f \in \on{Rad}^{lr}(\mathcal{C})$.
\end{proof}

\begin{theorem}\label{thm:mainthmmulti}
Let $\mathcal{C}$ be a finite pre-multitensor category. Then, $\mathcal{C}$ is Morita equivalent to a finite multitensor category if and only if $\mathcal{Z}(\mathcal{C})$ is rigid.
\end{theorem}
\begin{proof}
Let $\mathcal{C}$ be a discrete finite pre-multitensor category such that $\on{Rad}^{lr}(\mathcal{C})$ vanishes or, equivalently thanks to \autoref{prop:twosidedradicalcenter}, whose Drinfeld center is rigid.
In order to establish the backward direction of \autoref{mainthm}, we show that the exact left $\mathcal{C}$-module $\mathcal{C}/\on{Rad}^{\mathcal{C}}(\mathcal{C})$ witnesses a Morita equivalence between $\mathcal{C}$ and $\on{End}_{\mathcal{C}}^{\on{rex}}(\mathcal{C}/\on{Rad}^{\mathcal{C}}(\mathcal{C}))^{\on{mop}}$. The latter being a finite multitensor category by \autoref{prop:dualtensorcat}, this will conclude the proof.

By \autoref{lem:finitemodulegenerator},
there exists an algebra $A$ in $\mathcal{C}$ such that $\mathcal{C}/\on{Rad}^{\mathcal{C}}(\mathcal{C})\simeq\rmod$ as left $\mathcal{C}$-module categories.
It follows that $\on{End^{rex}_{\mathcal{C}}}(\mathcal{C}/\on{Rad}^{\mathcal{C}}(\mathcal{C}))^{\on{mop}}\simeq\bimod$.
We wish to appeal to \autoref{thm:generalizedMorita}. In order to do so, we must check that $A$ is faithful and that the underlying object of $A$ is $\mathcal{C}\boxtimes\mathcal{C}^{\on{mop}}$-projective. On the one hand, faithfulness of $A$ follows from \autoref{prop:Matmagic} via the Eilenberg-Watts theorem. On the other hand, that $A$ is $\mathcal{C}\boxtimes\mathcal{C}^{\on{mop}}$-projective is a consequence of the more general fact that $\mathcal{C}$ is exact as a $\mathcal{C}\boxtimes\mathcal{C}^{\on{mop}}$-module category, which is implied by \autoref{thm:CSZpretensor} and \autoref{prop:twosidedradicalcenter}.
\end{proof}

\begin{corollary}
Let $\mathcal{C}$ be a finite pre-multitensor category over an algebraically closed field. If the Drinfeld center of $\mathcal{C}$ is trivial, then $\mathcal{C}$ is Morita equivalent to $\mathrm{Vec}$. In particular, there is an equivalence of pre-multitensor categories $\mathcal{C}\simeq \bimod[\mathbbm{k}][A]$ for some finite dimensional $\mathbbm{k}$-algebra $A$.
\end{corollary}
\begin{proof}
If $\mathbbm{k}$ is algebraically closed, then $\mathrm{Vec}$ is the only finite tensor category whose Drinfeld center is $\mathrm{Vec}$. This follows for instance from \cite[Theorem 7.16.6]{EGNO}
\end{proof}

\begin{remark}
If $\mathbbm{k}$ is not algebraically closed, then the statement of \autoref{maincor1} may need to be modified depending on the ground field $\mathbbm{k}$. Namely, it was shown in \cite{SS} that the Morita equivalence classes of finite tensor categories with Drinfeld center $\mathrm{Vec}$ are classified by the third Galois cohomology group $H^3(\on{Gal}(\overline{\mathbbm{k}}/\mathbbm{k}),\overline{\mathbbm{k}}^\times)$.
\end{remark}

\section{Applications}

We now discuss some applications of our results to the study of the dualizability properties of (higher) Morita categories of (braided) pre-tensor categories.
Our main motivation stems from the cobordism hypothesis \cite{BD,L}, which posits an equivalence between \emph{fully dualizable} objects in a symmetric monoidal $n$-category $\mathscr{C}$ and $n$-dimensional topological field theories valued in $\mathscr{C}$.
In low dimensions, these topological field theories are expected to be related to classical constructions in quantum topology.
For instance, the Morita 3-category $\on{Mor}_1^{\on{tens}}$ of finite multitensor categories was extensively studied in \cite{DSPS:book}.
They prove that the fully dualizable objects in $\on{Mor}_1^{\on{tens}}$ are precisely \emph{separable} tensor categories, which are conjecturally closely related to the $3$-dimensional topological field theories obtained via the Turaev-Viro construction.
Using our main theorem, we characterize the fully dualizable objects of the Morita 3-category $\on{Mor}_1^{\on{pre}}$ of finite pre-multitensor categories.
Going up in dimension, there is a higher Morita 4-category $\on{Mor}_2^{\on{pre}}$ of finite braided pre-multitensor categories --- a variant of the Morita 4-category introduced in \cite{BJS}.
We give a necessary criterion for an object of $\on{Mor}_2^{\on{pre}}$ to be fully dualizable. This complements the sufficient criterion given in \cite{Dec:relative}.
Finally, in a somewhat different direction, we briefly discuss how our main theorem can be used to characterize Witt equivalence between (enriched) finite braided tensor categories thereby providing a non-semisimple version of \cite[Proposition 5.6]{DNO}.

\subsection{The Morita Theory of Finite Pre-Tensor Categories}\label{sub:Mor1}

Recall from \cite{Dec:relative} that there is a symmetric monoidal Morita 3-category $\on{Mor}_1^{\on{pre}}$ whose structure may be succinctly described as:
\begin{itemize}
\item Objects are finite pre-multitensor categories (over a fixed perfect field $\mathbbm{k}$);
\item 1-Morphisms are finite bimodule categories;
\item 2-Morphisms are right exact bimodule functors;
\item 3-Morphisms are bimodule natural transformations.
\end{itemize}
The symmetric monoidal structure on $\on{Mor}_1^{\on{pre}}$ is given by the Deligne tensor product.
In particular, the full subcategory $\on{Mor}_1^{\on{tens}}$ of $\on{Mor}_1^{\on{pre}}$ on the finite multitensor categories is symmetric monoidal. The symmetric monoidal 3-category $\on{Mor}_1^{\on{tens}}$ was extensively studied in \cite{DSPS:book}. We begin by giving a higher categorical interpretation of the notion of Morita equivalence between finite pre-multitensor categories.

\begin{lemma}\label{lem:Moritahighercat}
Two finite pre-multitensor categories are Morita equivalent if and only if they are equivalent as objects of the 3-category $\on{Mor}_1^{\on{pre}}$.
\end{lemma}
\begin{proof}
Recall, for instance from \cite[Lemma 2.24]{BJSS}, that a 1-morphism in a higher category is invertible if and only if it is adjunctible and the unit and counit 2-morphisms witnessing the adjunction are both equivalences.
But, it follows from \cite[Corollary 5.5]{BJS} (see also \cite[Section 3.2.1]{DSPS:book}) that any 1-morphism in $\on{Mor}_1^{\on{pre}}$, that is, any finite $\mathcal{C}$-$\mathcal{D}$-bimodule category $\mathcal{M}$, is adjunctible with unit and counit given by the functors
$$\mathcal{D}^{\mathrm{mop}}\rightarrow \on{End^{rex}_{\mathcal{C}}}(\mathcal{M})\quad\mathrm{and}\quad\mathcal{C}\rightarrow \on{End^{rex}_{\mathcal{D}}}(\mathcal{M})\,.$$
By definition, these two functors are equivalences if and only if $\mathcal{M}$ witnesses a Morita equivalence.
\end{proof}

It is well-known that Morita equivalence between $\mathbbm{k}$-algebras can be reformulated as an equivalence between the categories of modules. This fact generalizes to our setting. More precisely, given a finite pre-multitensor category $\mathcal{C}$, let us write $\mathbf{Mod}(\mathcal{C})$ for the 2-category of finite left $\mathcal{C}$-module categories and right exact $\mathcal{C}$-module functors.

\begin{corollary}\label{cor:Morita2catmodules}
Let $\mathcal{M}$ be a finite $\mathcal{C}$-$\mathcal{D}$-bimodule category witnessing a Morita equivalence between $\mathcal{C}$ and $\mathcal{D}$. Then, there is an equivalence of 2-categories
$$\mathcal{M}\boxtimes_{\mathcal{D}}-:\mathbf{Mod}(\mathcal{D})\xrightarrow{\simeq} \mathbf{Mod}(\mathcal{C})\,.$$
If, in addition, $\mathcal{C}$ or $\mathcal{D}$ is discrete, the above restricts to an equivalence between the 2-categories of exact module categories.
\end{corollary}
\begin{proof}
The first part follows from the observation that $\mathbf{Mod}(\mathcal{C}) = \on{Hom}_{\on{Mor}_1^{\on{pre}}}(\mathcal{C},\mathrm{Vec})$. The second part is a consequence of \autoref{cor:characterizationexactness}.
\end{proof}

In a different direction, recall that an object of a symmetric monoidal $3$-category is $3$-dualizable if it is left and right rigid, the corresponding evaluation and coevaluation 1-morphisms have all possible adjoints, and the corresponding unit and counit 2-morphisms have all possible adjoints. We refer the reader to \cite[Chapter 1]{DSPS:book} for additional details and illuminating pictures.
Fully dualizable objects in $\on{Mor}_1^{\on{tens}}$ are completely characterized via the following algebraic property.

\begin{definition}
A finite multitensor category $\mathcal{C}$ is \emph{separable} if its Drinfeld center $\mathcal{Z}(\mathcal{C})$ is finite semisimple.
\end{definition}

\begin{theorem}[\cite{DSPS:book}]
A finite multitensor category $\mathcal{C}$ is fully dualizable as an object of $\on{Mor}_1^{\on{tens}}$ if and only if it is separable.
\end{theorem}

Thanks to our main theorem, we obtain a complete characterization of the fully dualizable objects in $\on{Mor}_1^{\on{pre}}$. More precisely, it turns out that they all reside within the essential image of $\on{Mor}_1^{\on{tens}}\hookrightarrow\on{Mor}_1^{\on{pre}}$. By way of \autoref{lem:Moritahighercat}, this follows from \autoref{cor:fdpre}, which we now prove.

\begin{theorem}\label{thm:fdMorita3catpre}
A finite pre-multitensor category is fully dualizable as an object of $\mathrm{Mor_1^{pre}}$ if and only if its Drinfeld center is a multifusion category if and only if it is Morita equivalent to a separable multitensor category.
\end{theorem}

\begin{proof}
Let $\mathcal{C}$ be a finite pre-multitensor category that is fully dualizable as an object of $\mathrm{Mor}_1^{\on{pre}}$.
As in \cite[Section 3.2.2]{DSPS:book}, the underlying category of $\mathcal{Z}(\mathcal{C})$ can be expressed as the composite of two fully adjunctible 1-morphisms
\begin{equation}\label{eq:fullyadjunctible1morphisms}
\mathcal{Z}(\mathcal{C})\simeq\on{End}^{\on{rex}}_{\mathcal{C}\boxtimes\mathcal{C}^{\on{mop}}}(\mathcal{C}) \simeq\on{Fun}^{\on{rex}}_{\mathcal{C}\boxtimes\mathcal{C}^{\on{mop}}}(\mathcal{C},\mathcal{C}\boxtimes\mathcal{C}^{\on{mop}})\boxtimes_{\mathcal{C}\boxtimes\mathcal{C}^{\on{mop}}}\mathcal{C}\,.
\end{equation}
In particular, we find that $\mathcal{Z}(\mathcal{C})$ is a fully adjunctible 1-morphism from $\mathrm{Vec}$ to $\mathrm{Vec}$ in $\mathrm{Mor}_1^{\on{pre}}$.
Said differently, the underlying category of $\mathcal{Z}(\mathcal{C})$ is fully dualizable as an object of the symmetric monoidal 2-category of finite categories.
It follows from \cite[Theorem A.22]{BDSPV} that $\mathcal{Z}(\mathcal{C})$ is finite semisimple.
We claim that $\mathcal{Z}(\mathcal{C})$ is a pre-multitensor category. Given that $\mathcal{Z}(\mathcal{C})$ is finite semisimple, it follows that it is rigid. Thanks to \autoref{mainthm}, this will conclude the proof.

It remains to show that $\mathcal{Z}(\mathcal{C})$ is a pre-multitensor category.
Thanks to \autoref{rem:BJSrigid}, we will equivalently check that the monoidal product functor $T:\mathcal{Z}(\mathcal{C})\boxtimes\mathcal{Z}(\mathcal{C})\rightarrow\mathcal{Z}(\mathcal{C})$ has a right adjoint as a bimodule functor.
Using the equivalence of \autoref{eq:fullyadjunctible1morphisms} above, the functor $T$ is induced by the $\mathcal{C}\boxtimes\mathcal{C}^{\on{mop}}$-bimodule functor 
$$E:\mathcal{C}\boxtimes\on{Fun}^{\on{rex}}_{\mathcal{C}\boxtimes\mathcal{C}^{\on{mop}}}(\mathcal{C},\mathcal{C}\boxtimes\mathcal{C}^{\on{mop}})\rightarrow\mathcal{C}\boxtimes\mathcal{C}^{\on{mop}}$$
serving as the counit of the adjunction between the two 1-morphisms of \autoref{eq:fullyadjunctible1morphisms}. Explicitly, the functor $E$ is given by evaluation.
More precisely, we have that $$T := \on{Fun}^{\on{rex}}_{\mathcal{C}\boxtimes\mathcal{C}^{\on{mop}}}(\mathcal{C},\mathcal{C}\boxtimes\mathcal{C}^{\on{mop}})\boxtimes_{\mathcal{C}\boxtimes\mathcal{C}^{\on{mop}}} E\boxtimes_{\mathcal{C}\boxtimes\mathcal{C}^{\on{mop}}}\mathcal{C}\,.$$
But, as $\mathcal{C}$ is fully dualizable by assumption, $E$ must have a right exact right adjoint $E^{\on{R}}$ as a $\mathcal{C}\boxtimes\mathcal{C}^{\on{mop}}$-bimodule functor. It follows that $T$ has a right exact right adjoint $T^{\on{R}}:\mathcal{Z}(\mathcal{C})\rightarrow\mathcal{Z}(\mathcal{C})\boxtimes\mathcal{Z}(\mathcal{C})$ defined by
$$T^{\on{R}}:=\on{Fun}^{\on{rex}}_{\mathcal{C}\boxtimes\mathcal{C}^{\on{mop}}}(\mathcal{C},\mathcal{C}\boxtimes\mathcal{C}^{\on{mop}})\boxtimes_{\mathcal{C}\boxtimes\mathcal{C}^{\on{mop}}} E^{\on{R}}\boxtimes_{\mathcal{C}\boxtimes\mathcal{C}^{\on{mop}}}\mathcal{C}\,.$$
Unpacking definitions, we find that there are natural isomorphisms
\begin{equation}\label{eq:ZCmodule}
(\mathcal{Z}(\mathcal{C})\boxtimes T)\circ(T^{\on{R}}\boxtimes\mathcal{Z}(\mathcal{C})) \cong T^{\on{R}}\circ T\cong (T\boxtimes\mathcal{Z}(\mathcal{C}))\circ(\mathcal{Z}(\mathcal{C})\boxtimes T^{\on{R}})\,.
\end{equation}
Plainly, the two natural isomorphisms above arise from the following isomorphisms induced by the universal property of the relative Deligne tensor product:
$$\Big(\mathcal{C}\boxtimes\on{Fun}^{\on{rex}}_{\mathcal{C}\boxtimes\mathcal{C}^{\on{mop}}}(\mathcal{C},\mathcal{C}\boxtimes\mathcal{C}^{\on{mop}})\boxtimes_{\mathcal{C}\boxtimes\mathcal{C}^{\on{mop}}}E\Big)\circ \Big(E^{\on{R}}\boxtimes_{\mathcal{C}\boxtimes\mathcal{C}^{\on{mop}}}\mathcal{C}\boxtimes\on{Fun}^{\on{rex}}_{\mathcal{C}\boxtimes\mathcal{C}^{\on{mop}}}(\mathcal{C},\mathcal{C}\boxtimes\mathcal{C}^{\on{mop}})\Big)\cong E^{\on{R}}\circ E\,,$$
$$E^{\on{R}}\circ E\cong \Big(E\boxtimes_{\mathcal{C}\boxtimes\mathcal{C}^{\on{mop}}}\mathcal{C}\boxtimes\on{Fun}^{\on{rex}}_{\mathcal{C}\boxtimes\mathcal{C}^{\on{mop}}}(\mathcal{C},\mathcal{C}\boxtimes\mathcal{C}^{\on{mop}})\Big)\circ \Big(\mathcal{C}\boxtimes\on{Fun}^{\on{rex}}_{\mathcal{C}\boxtimes\mathcal{C}^{\on{mop}}}(\mathcal{C},\mathcal{C}\boxtimes\mathcal{C}^{\on{mop}})\boxtimes_{\mathcal{C}\boxtimes\mathcal{C}^{\on{mop}}}E^{\on{R}}\Big)\,.$$
It is straightforward to check that the isomorphisms of \autoref{eq:ZCmodule} are coherent, so that they supply $T^{\on{R}}$ with the structure of a $\mathcal{Z}(\mathcal{C})$-bimodule functor as desired. Furthermore, one checks via a similar argument that the unit and counit of the adjunction between $T$ and $T^{\on{R}}$ are compatible with these bimodule structures.
Alternatively, the natural isomorphisms of \autoref{eq:ZCmodule} are instances of the interchange law for 2-morphisms in a 3-category, so that the bimodule coherence axioms follows from the coherence axioms of a 3-category.
\end{proof}

\subsection{The Higher Morita Theory of Finite Braided Pre-Tensor Categories}

We begin by recalling from \cite{Dec:relative} a variant of the higher Morita 4-category studied in \cite{BJS}. Namely, there is a symmetric monoidal Morita 4-category $\on{Mor}_2^{\on{pre}}$ whose structure may be succinctly described as:
\begin{itemize}
\item Objects are finite braided pre-multitensor categories;
\item 1-Morphisms are finite central pre-multitensor categories;
\item 2-Morphisms are finite central bimodule categories;
\item 3-Morphisms are right exact centered bimodule functors;
\item 4-Morphisms are bimodule natural transformations.
\end{itemize}
More precisely, given two finite braided pre-multitensor categories $\mathcal{A}$ and $\mathcal{B}$, a 1-morphisms consists of a finite pre-multitensor category $\mathcal{C}$ equipped with a right exact braided monoidal functor $\mathcal{A}\boxtimes\mathcal{B}^{\on{rev}}\rightarrow\mathcal{Z}(\mathcal{C})$. We also say that $\mathcal{C}$ is an $\mathcal{A}$-$\mathcal{B}$-central finite pre-multitensor category.
We refer the reader to \cite[Section 3]{BJS} for a thorough unpacking of the structure of 2- and 3-morphisms.
The dualizability properties of $\on{Mor}_2^{\on{pre}}$ have been extensively studied in \cite{BJSS,BJS}. The most general result concerning fully dualizability is the following.

\begin{theorem}[\cite{Dec:relative}]
A finite braided multitensor category $\mathcal{B}$ is fully dualizable as an object of $\on{Mor}_2^{\on{pre}}$ if its symmetric center $\mathcal{Z}_{(2)}(\mathcal{B})$ is separable.
\end{theorem}

Our main theorem allows to give a necessary condition on a finite braided pre-multitensor category to be fully dualizable as an object of $\on{Mor}_2^{\on{pre}}$.

\begin{proposition}
Let $\mathcal{B}$ be a finite braided pre-multitensor category.
If $\mathcal{B}$ is fully dualizable as an object of $\mathrm{Mor_2^{pre}}$, then its symmetric center is separable.
\end{proposition}

\begin{proof}
Let $\mathcal{B}$ be a finite braided pre-multitensor category. Assume that $\mathcal{B}$ is fully dualizable as an object of $\mathrm{Mor}_2^{\on{pre}}$.
It therefore follows, for instance from \cite[Theorem 5.16.1]{BJS}, that $$\mathrm{HC}(\mathcal{B}):=\mathcal{B}\boxtimes_{\mathcal{B}\boxtimes\mathcal{B}^{\on{rev}}}\mathcal{B}^{\on{mop}}$$
is fully dualizable as an object of $\mathrm{Mor}_1^{\on{pre}}$.
Thanks to \autoref{thm:fdMorita3catpre}, we have that $\mathrm{HC}(\mathcal{B})$ is Morita equivalent to a separable multitensor category. We will employ this observation below.

Given that $\mathcal{B}$ is a braided pre-tensor category, we may canonically view $\mathcal{B}$ as a braided $\mathcal{B}$-module category.
In particular, we find by direct inspection that $\on{End}^{\on{rex,br}}_{\mathcal{B}}(\mathcal{B})$, the category of right exact \emph{braided} $\mathcal{B}$-module endofunctors on $\mathcal{B}$, is identified with the symmetric center of $\mathcal{B}$:
$$\on{End}^{\on{rex,br}}_{\mathcal{B}}(\mathcal{B})\simeq\mathcal{Z}_{(2)}(\mathcal{B})$$
On the other hand, it follows from \cite[Theorem 3.11]{BZBJ} that there is an equivalence of 2-categories
$$\mathbf{Mod}^{\on{br}}(\mathcal{B})\simeq\mathbf{Mod}(\on{HC}(\mathcal{B}))\,,$$
between finite braided $\mathcal{B}$-module categories and finite $\on{HC}(\mathcal{B})$-module categories.
Specifically, associated to the canonical braided $\mathcal{B}$-module category $\mathcal{B}$ is the finite $\mathrm{HC}(\mathcal{B})$-module category $\mathcal{B}$ with action given by the tensor functor $\mathrm{HC}(\mathcal{B})\rightarrow\mathcal{B}$ induced by the tensor product functor $\mathcal{B}\boxtimes\mathcal{B}^{\on{mop}}\rightarrow\mathcal{B}$. We claim that $\mathcal{B}$ is an exact $\mathrm{HC}(\mathcal{B})$-module category.
As $\mathrm{HC}(\mathcal{B})$ is Morita equivalent to a separable multitensor category, it follows from \autoref{cor:Morita2catmodules} that $\on{End}_{\mathrm{HC}(\mathcal{B})}^{\on{rex}}(\mathcal{B})$ is a separable multitensor category.
Given that
$$\on{End}^{\on{rex}}_{\mathrm{HC}(\mathcal{B})}(\mathcal{B})\simeq\on{End}^{\on{rex,br}}_{\mathcal{B}}(\mathcal{B})\simeq\mathcal{Z}_{(2)}(\mathcal{B})\,,$$ this implies that $\mathcal{Z}_{(2)}(\mathcal{B})$ is a separable multitensor category as desired.

It remains to show that $\mathcal{B}$ is an exact $\mathrm{HC}(\mathcal{B})$-module category.
Viewing $\mathcal{B}$ as a left $\mathcal{B}\boxtimes\mathcal{B}^{\on{rev}}$-module category, we can consider the algebra $A:=\underline{\on{End}}(\mathbbm{1})$ in $\mathcal{B}\boxtimes\mathcal{B}^{\on{rev}}$. This algebra is such that $\mathcal{B}\simeq \rmod[\mathcal{B}\boxtimes\mathcal{B}^{\on{rev}}][A]$. Moreover, by construction, the object $\mathbbm{1}\in\mathcal{B}$ corresponds to the right $A$-module $A$.
We can therefore identify the functor $\mathrm{HC}(\mathcal{B})\rightarrow\mathcal{B}$ with the functor
$$\rmod[\mathcal{B}\boxtimes\mathcal{B}^{\on{rev}}][\,(A^{\on{op}}\otimes A)]\rightarrow\rmod[\mathcal{B}\boxtimes\mathcal{B}^{\on{rev}}][A]$$
corresponding to the $A^{\on{op}}\otimes A$-$A$-bimodule $A$.
This functor sends projective $A^{\on{op}}\otimes A$-modules to projective $A$-modules, so $\mathcal{B}$ is an exact $\mathrm{HC}(\mathcal{B})$-module category. This concludes the proof.
\end{proof}

In particular, a complete characterization of the fully dualizable objects of $\on{Mor}_2^{\on{pre}}$ seems within reach. Namely, it only remains to answer the following question.

\begin{question}
Let $\mathcal{B}$ be a braided finite pre-multitensor category. If $\mathcal{B}$ is fully dualizable as an object of $\on{Mor}_2^{\on{pre}}$, is $\mathcal{B}$ necessarily a \emph{multitensor} category?
\end{question}

\subsection{Witt Equivalence between Enriched Finite Braided Tensor Categories}

Fix $\mathcal{E}$ a finite symmetric (multi)tensor category. A \emph{faithfully flat} $\mathcal{E}$-enriched finite multitensor category is a multitensor category equipped with a fully faithful tensor embedding $\mathcal{E}\hookrightarrow\mathcal{Z}(\mathcal{C})$. If $\mathcal{C}$ is a faithfully flat $\mathcal{E}$-enriched finite braided multitensor category, we write $\mathcal{Z}(\mathcal{C},\mathcal{E})$ for the centralizer of $\mathcal{E}\hookrightarrow\mathcal{Z}(\mathcal{C})$, that is, $\mathcal{Z}(\mathcal{C},\mathcal{E})$ is the full tensor subcategory of $\mathcal{Z}(\mathcal{C})$ on those objects that double braid trivially with $\mathcal{E}$. An $\mathcal{E}$-non-degenerate finite braided multitensor category is a finite braided multitensor category equipped with an equivalence $\mathcal{E}\xrightarrow{\sim}\mathcal{Z}_{(2)}(\mathcal{B})$. The next definition is the non-semisimple generalization of \cite[Definition 5.1]{DNO}.

\begin{definition}
Let $\mathcal{A}$ and $\mathcal{B}$ be $\mathcal{E}$-non-degenerate finite braided multitensor categories. We say that $\mathcal{A}$ and $\mathcal{B}$ are \emph{Witt equivalent} if there exists faithfully flat $\mathcal{E}$-enriched finite multitensor categories $\mathcal{C}$ and $\mathcal{D}$ and an equivalence of braided multitensor categories 
\begin{equation}\label{eq:Wittequivalence}
\mathcal{A}\boxtimes_{\mathcal{E}}\mathcal{Z}(\mathcal{C},\mathcal{E})\simeq\mathcal{B}\boxtimes_{\mathcal{E}}\mathcal{Z}(\mathcal{D},\mathcal{E})\,.
\end{equation}
\end{definition}

\begin{proposition}\label{prop:Wittequivalence}
Let $\mathcal{A}$ be an $\mathcal{E}$-non-degenerate finite braided multitensor category. The Witt equivalence class of $\mathcal{A}$ is trivial if and only if there exists a faithfully flat $\mathcal{E}$-enriched finite multitensor category $\mathcal{D}$ such that $\mathcal{A}\simeq\mathcal{Z}(\mathcal{D},\mathcal{E})$, as finite braided multitensor categories.
\end{proposition}
\begin{proof}
The backward direction is immediate.
Conversely, if the Witt class of $\mathcal{A}$ is trivial, we have by definition
\begin{equation}\label{eq:technicalWitt}
\mathcal{A}\boxtimes_{\mathcal{E}}\mathcal{Z}(\mathcal{C},\mathcal{E})\simeq\mathcal{Z}(\mathcal{D},\mathcal{E})\,.
\end{equation}
Let $L$ be the commutative exact algebra in $\mathcal{Z}(\mathcal{C},\mathcal{E})$ such that $\rmod[\mathcal{Z}(\mathcal{C},\mathcal{E})][L]\simeq \mathcal{C}$ as finite multitensor categories.
We write
$\modloc[{\mathcal{Z}(\mathcal{C},\mathcal{E})}][L]$ for the finite braided multitensor category of \emph{local}, also called dyslectic, $L$-modules in $\mathcal{Z}(\mathcal{C},\mathcal{E})$. By construction of $L$, we have $\modloc[\mathcal{Z}(\mathcal{C},\mathcal{E})][L]\simeq\mathcal{E}$.
Consequently, upon taking categories of local $L$-module on the two sides of \autoref{eq:technicalWitt}, we find
$$
\mathcal{A}\simeq \modloc[\mathcal{Z}(\mathcal{D},\mathcal{E})][L]\,.
$$
Recall from \cite[Corollary 4.5]{Sch}, that there is an equivalence of finite braided multitensor categories
$$\modloc[\mathcal{Z}(\mathcal{D})][L]\simeq \mathcal{Z}(\rmod[\mathcal{D}][L])\,.$$
As commutative exact algebras are preserved by fully faithful braided tensor functors, see \cite[Proof of Proposition 2.1.3]{Dec:relative}, it follows from \cite[Theorem 5.5]{SY} that $\modloc[\mathcal{Z}(\mathcal{D})][L]$ is rigid.
Furthermore, as in \cite[Proposition 4.5]{DNO}, we deduce an equivalence of finite braided multitensor categories
$$\modloc[\mathcal{Z}(\mathcal{D},\mathcal{E})][L]\simeq \mathcal{Z}(\rmod[\mathcal{D}][L],\mathcal{E})\,.$$
Thus, the proof will be completed provided that we can show that $\mathcal{Z}(\rmod[\mathcal{D}][L])$ is the Drinfeld center of a finite \emph{multitensor} category.
But we have already seen that $\mathcal{Z}(\rmod[\mathcal{D}][L])$ is rigid, so that \autoref{thm:mainthmmulti} implies that the finite pre-multitensor category $\rmod[\mathcal{D}][L]$ is Morita equivalent to a finite multitensor category. Thanks to \autoref{prop:centerMorita}, it follows that $\mathcal{Z}(\rmod[\mathcal{D}][L])$ is the Drinfeld center of a finite multitensor category as desired. This last step can also be deduced from \cite[Corollary 5.16]{OU}, and finishes the proof.
\end{proof}

\begin{remark}
Taking $\mathcal{E}=\mathrm{Vec}$ in \autoref{prop:Wittequivalence}, we recover \cite[Proposition 5.17]{OU}, which positively answered \cite[Question 7.21]{SY}.
In fact, thanks to \autoref{thm:mainthmmulti}, we obtain a higher categorical refinement of this observation.
Namely, it follows from our main theorem that a finite braided multitensor category is trivial, that is, equivalent to $\mathrm{Vec}$, as an object of $\mathrm{Mor}_2^{\on{pre}}$ if and only if it is the Drinfeld center of a finite multitensor category.
This answers positively a variant of \cite[Question 4.4]{BJSS}. As a consequence, the Picard group of $\mathrm{Mor}_2^{\on{pre}}$, that is, its group of equivalence classes of invertible objects, coincides with the group of Witt equivalence class of finite braided tensor categories. We warn the reader that these last observations do not, at least a priori, extend to the case when $\mathcal{E}\not\simeq\mathrm{Vec}$. In particular, we do not provide an answer to \cite[Conjecture 4.2.4]{Dec:relative} --- in order to do so, we would need to prove \autoref{prop:Wittequivalence} assuming only that $\mathcal{D}$ is finite pre-multitensor.
\end{remark}

%%%%%%%%%%%%%%%%%%%%%%%%%%%%%%%%%%%%%%%%%%%%%%%%%%%%%%%%%%%%%

\end{document}